\documentclass[11pt]{article}
\usepackage{fixme}
\usepackage[top=3cm,bottom=3cm,margin=2.5cm]{geometry}

\usepackage{sseq}

\usepackage{amsmath,amssymb,amsthm,mathrsfs}
\usepackage{stmaryrd}
\usepackage{epsfig}
\usepackage{psfrag}
\usepackage[matrix, arrow, curve]{xy}

\usepackage[mathscr]{euscript}

\usepackage{yfonts}
\usepackage{calligra}
\usepackage{graphicx}
\usepackage{comment}
\usepackage{dsfont}

\usepackage{enumitem}
\usepackage{etoolbox}
\AtBeginEnvironment{theorem}{\setlist[enumerate,1]{label=(\arabic*),font=\upshape}}
\AtBeginEnvironment{proposition}{\setlist[enumerate,1]{label=(\arabic*),font=\upshape}}
\AtBeginEnvironment{lemma}{\setlist[enumerate,1]{label=(\arabic*),font=\upshape}}
\AtBeginEnvironment{corollary}{\setlist[enumerate,1]{label=(\arabic*),font=\upshape}}
\AtBeginEnvironment{definition}{\setlist[enumerate,1]{label=(\arabic*),font=\upshape}}
\AtBeginEnvironment{remark}{\setlist[enumerate,1]{label=(\arabic*),font=\upshape}}
\AtBeginEnvironment{notation}{\setlist[enumerate,1]{label=(\arabic*),font=\upshape}}

\usepackage{lmodern}
\usepackage{xifthen}
\usepackage[final]{hyperref}

\theoremstyle{plain} 
\newtheorem{theorem}{Theorem}[section]
\newtheorem{corollary}[theorem]{Corollary}
\newtheorem{proposition}[theorem]{Proposition}
\newtheorem{lemma}[theorem]{Lemma}
\theoremstyle{definition} 
\newtheorem{definition}[theorem]{Definition}
\newtheorem{example}[theorem]{Example}

\newtheorem{remark}[theorem]{Remark}
\newtheorem{notation}[theorem]{Notation}

\numberwithin{equation}{section}

\newcommand{\A}{\mathscr{A}}
\newcommand{\B}{\mathscr{B}}
\newcommand{\C}{\mathscr{C}}
\newcommand{\D}{\mathscr{D}}
\newcommand{\E}{\mathscr{E}}
\newcommand{\F}{\mathscr{F}}
\newcommand{\G}{\mathscr{G}}
\renewcommand{\H}{\mathscr{H}}

\newcommand{\K}{\mathscr{K}}

\newcommand{\M}{\mathscr{M}}
\newcommand{\N}{\mathscr{N}}
\renewcommand{\O}{\mathscr{O}}
\renewcommand{\P}{\mathscr{P}}
\newcommand{\Q}{\mathscr{Q}}
\newcommand{\R}{\mathscr{R}}
\renewcommand{\S}{\mathscr{S}}
\newcommand{\T}{\mathscr{T}}
\newcommand{\U}{\mathscr{U}}

\renewcommand{\AA}{\mathbb{A}}

\newcommand{\CC}{\mathbb{C}}

\newcommand{\EE}{\mathbb{E}}

\newcommand{\NN}{\mathbb{N}}

\newcommand{\ZZ}{\mathbb{Z}}

\newcommand{\GGG}{\mathbf{G}}
\newcommand{\HHH}{\mathbf{H}}

\newcommand{\coh}{\mathrm{coh}}
\newcommand{\qcc}{\mathrm{qcc}}
\newcommand{\qcoh}{\mathrm{qcoh}}
\newcommand{\gqcoh}{\mathrm{gqcoh}}
\newcommand{\Chbar}{{\mathbb{C}^\hbar}}
\DeclareMathOperator{\gr}{gr_\hbar}
\DeclareMathOperator{\cogr}{cogr_\hbar}

\newcommand{\Ltensor}{\mathop{\otimes}\limits^{\mathrm{L}}}
\newcommand{\Lutensor}{\mathop{\underline{\otimes}}\limits^{\mathrm{L}}}
\newcommand{\tensor}{\mathop{\otimes}\limits^{}}
\newcommand{\btimes}{\mathop{\boxtimes}\limits^{}}
\newcommand{\ubtimes}{\mathop{\underline{\boxtimes}}\limits^{}}
\newcommand{\ctensor}{\mathop{\widehat{\otimes}}\limits^{}}
\newcommand{\cubtimes}{\mathop{\underline{\widehat{\boxtimes}}}\limits^{}}
\newcommand{\utensor}{\mathop{\underline{\otimes}}\limits^{}}
\newcommand{\cutensor}{\mathop{\underline{\widehat{\otimes}}}\limits^{}}

\newcommand{\convp}{\mathop{\circ}\limits^{}}

\newcommand{\dsu}{\mathds{1}}
\newcommand{\Der}{\mathscr{D}}
\newcommand{\Derb}{\mathscr{D}^\mathrm{b}}
\newcommand{\hDer}{\mathrm{D}}
\newcommand{\hDerb}{\mathrm{D}^\mathrm{b}}

\renewcommand{\i}{\infty}

\newcommand{\too}{\longrightarrow}

\newcommand{\op}{\mathrm{op}}

\newcommand{\fExt}{\E xt}
\newcommand{\cc}{\mathrm{cc}}
\newcommand{\Cpl}{\mathrm{cpl}}
\newcommand{\loc}{\mathrm{loc}}
\newcommand{\nil}{\mathrm{nil}}

\newcommand{\inj}{\mathrm{inj}}

\newcommand{\pullbackcorner}[1][dl]{\save*!/#1-1pc/#1:(-1,-1)@_{|-}\restore}

\DeclareMathOperator{\fEnd}{\mathscr{E}nd}
\DeclareMathOperator{\fHom}{\mathscr{H}om}

\DeclareMathOperator{\OfFun}{\F un}
\DeclareMathOperator{\OsFun}{\mathbf{Fun}}

\DeclareMathOperator{\smf}{sf}
\DeclareMathOperator{\coker}{coker}
\DeclareMathOperator{\Perf}{Perf}

\DeclareMathOperator{\Hn}{H}
\DeclareMathOperator{\sMod}{\mathscr{M}od}

\DeclareMathOperator{\Mod}{Mod}

\DeclareMathOperator{\Cat}{Cat}

\DeclareMathOperator{\OFib}{Fib}

\DeclareMathOperator{\OFun}{Fun}
\newcommand{\Funl}{\mathrm{Fun}^\mathrm{L}}
\newcommand{\Funr}{\mathrm{Fun}^\mathrm{R}}

\DeclareMathOperator{\Cyl}{Cyl}

\DeclareMathOperator{\id}{id}

\DeclareMathOperator{\Map}{Map}
\DeclareMathOperator{\uMap}{\underline{Map}}
\DeclareMathOperator{\fMap}{\M ap}

\DeclareMathOperator{\colim}{colim}

\DeclareMathOperator{\Open}{Open}

\DeclareMathOperator{\OPr}{Pr}

\DeclareMathOperator{\OfPPStk}{2-\P \S tk}

\DeclareMathOperator{\Gr}{Gr}
\DeclareMathOperator{\Ch}{Ch}
\DeclareMathOperator{\Fct}{Fun}

\newcommand{\Prr}[1][]{
                                        \ifthenelse{\isempty{#1}}%
                                        {\OPr^\mathrm{R}}%
                                        {\OPr^\mathrm{R}_{#1}}%
                                       }
\newcommand{\Prl}[1][]{
                                        \ifthenelse{\isempty{#1}}%
                                        {\OPr^\mathrm{L}_{\mathrm{st}}}%
                                        {\OPr^\mathrm{L}_{#1}}%
                                       }
\newcommand{\Fib}[1][]{
                                        \ifthenelse{\isempty{#1}}%
                                        {\OFib}%
                                        {\OFib_{\,#1}}%
                                      }

\newcommand{\fPPStk}[1][]{
                                        \ifthenelse{\isempty{#1}}%
                                        {\OfPPStk}%
                                        {\OfPPStk_{#1}}%
                                       }
\newcommand{\Fun}[1][]{
                                        \ifthenelse{\isempty{#1}}%
                                        {\OFun}%
                                        {\OFun_{#1}}%
                                       }

\newcommand{\fFun}[1][]{
                                        \ifthenelse{\isempty{#1}}%
                                        {\OfFun}%
                                        {\OfFun_{#1}}%
                                       }

\newcommand{\sFun}[1][]{
                                        \ifthenelse{\isempty{#1}}%
                                        {\OsFun}%
                                        {\OsFun_{#1}}%
                                       }

\newcommand{\comp}{i^\land\,}

\DeclareMathOperator{\PrL}{\Prl[st]}

\DeclareMathOperator{\Shv}{Shv}

\renewcommand{\id}{\mathrm{id}}

\def\fibdown{\ar@{->>}[d]}
\def\hookdown{\ar@<-.5ex>[d]|{\phantom{a}}|<<{\put(-.7,2){$\scriptstyle\cap$}}}

\def\llarrow{  \hspace{.05cm}\mbox{\,\put(0,-2){$\leftarrow$}\put(0,2){$\leftarrow$}\hspace{.45cm}}}
\def\rrarrow{  \hspace{.05cm}\mbox{\,\put(0,-2){$\rightarrow$}\put(0,2){$\rightarrow$}\hspace{.45cm}}}
\def\lllarrow{ \hspace{.05cm}\mbox{\,\put(0,-3){$\leftarrow$}\put(0,1){$\leftarrow$}\put(0,5){$\leftarrow$}\hspace{.45cm}}}
\def\rrrarrow{ \hspace{.05cm}\mbox{\,\put(0,-3){$\rightarrow$}\put(0,1){$\rightarrow$}\put(0,5){$\rightarrow$}\hspace{.45cm}}}

\renewcommand{\epsilon}{\varepsilon}

\DeclareMathOperator{\Ind}{Ind}

\title{Integral representation theorems for DQ-modules}
\author{David Gepner \& Fran\c{c}ois Petit}

\newcommand{\Addresses}{{
  \bigskip
  \footnotesize

\noindent  D.~Gepner, \textsc{School of Mathematics and Statistics, The University of Melbourne, Parkville,\\ VIC, 3010 Australia}\par\nopagebreak
\noindent  \textit{E-mail address}: \texttt{david.gepner@unimelb.edu.au}

  \medskip

\noindent  F.~Petit, \textsc{Université de Paris, CRESS, INSERM, INRA, F-75004 Paris, France}\par\nopagebreak
\noindent  \textit{E-mail address}: \texttt{francois.petit@u-paris.fr}
}}

\begin{document}

\maketitle

\begin{abstract}
We identify the type of $\CC[[\hbar]]$-linear structure inherent in the $\i$-categories which arise in the theory of Deformation Quantization modules. Using this structure, we show that the $\i$-category of quasicoherent cohomologically complete DQ-modules is a deformation of the $\i$-category of quasicoherent sheaves. We also obtain integral representation results for DQ-modules similar to the ones of To\"en and Ben-Zvi-Nadler-Francis, stating that suitably linear functors between $\i$-categories of DQ-modules are integral transforms.  
\end{abstract}

\tableofcontents

\section{Introduction}

One of the major developments in algebraic geometry towards the end of the last century was the study of {\em Fourier-Mukai transforms}, also known as {\em integral transforms}.
More precisely, the Fourier-Mukai transform is a sheaf-theoretic analogue of the notion of integral transform.
If $X$ and $Y$ are smooth manifolds and $K(x,y)$ is a function on $X \times Y$, then (under reasonable hypotheses) the integral transform of a function $f(x)$ on $X$ is the function $T(f)$ on $Y$ given by the formula
\[
T(f)(y):= \int f(x) K(x,y) dx.
\]
Writing $p\colon X \times Y \to X$ and $q : X \times Y \to Y$ for the projections, this formula can be rewritten
\[
T(f)(y)=\int (f \circ p)(x,y) K(x,y) dx =\int p^*(f) K(x,y) dx,
\]
and the integral along the $x$ variable amounts to the pushforward along $q$; that is,
\[
T(f)(y)=q_*( p^*(f) K(x,y)).
\]
This is the idea of integral transforms in the setting of sheaf theory.

Historically, one of the first notable examples of an integral transform was the Fourier-Sato transform, introduced in \cite{SKK}, and further studied by several authors, notably Kashiwara-Schapira \cite{KS1} in the context of the microlocal theory of sheaves. In algebraic geometry, inspired by the formalism of integral transforms, Mukai constructed an equivalence between the bounded derived categories of coherent sheaves on an abelian variety and the dual abelian variety \cite{Mukai}.
This work established the importance of the Fourier-Mukai  transform in algebraic geometry.

In the 2000s, the conceptual approach to integral representation theorems changed, especially under the influence of the article \cite{mordg}. In this paper, B. To\" en develops the derived Morita theory of dg-categories and relying on the compact generation result for quasi-coherent sheaves of A. Bondal and M. Van den Bergh \cite{BVdB}, set integral representation theorems for quasi-coherent sheaves in this context. These ideas where further developed in \cite{BFN,Ben-Zvi2013} in the framework of linear $\i$-categories and allowed to obtain far reaching generalizations of Orlov's theorem \cite{Orlov1997}, one of the central result in the study of the structure of the derived category of coherent sheaves on a smooth algebraic variety. This theorem and its refinement state that any exact fully faithful functor between derived categories of coherent sheaves on smooth projective varieties is given by an integral transform such that the kernel is an object of the derived category of coherent sheaves of the product variety.

This project is rooted in the line of thought, promoted since the 1990s by the Moscow school, that the derived category of coherent sheaves of an algebraic variety should be considered as an invariant of the latter. More precisely, we'd like to know the extent to which derived categories of DQ-modules can be regarded as invariants of Poisson varieties, as DQ-algebroids quantize Poisson varieties.  A fundamental step in such a program is to describe functors between such categories.

The study of Fourier-Mukai transforms for DQ-modules was initiated in \cite{NCmukai}. This study was further pursued  by several authors \cite{KS3, Mukai_Pantev, FMQ}. In particular in \cite{KS3}, the authors establish several fundamental finiteness and duality results for DQ-modules and lay the foundations for integral transforms in this context. To that end, they introduce the notion of cohomological completeness, a notion closely related to derived adic-completion (see \cite{SAG} and \cite{Porta2012} for a comparison of the two notions), and defined in term of semi-orthogonal decomposition of the triangulated categories of DQ-modules. This notion of cohomological completion was vastly generalised by many authors, culminating in the treatment of J. Lurie in \cite{DAGXII, SAG}, and plays a key role in our work.

Using the theory of linear $\i$-categories, we establish an integral representation theorem for colimit preserving functors between $\i$-categories of DQ-modules. We also prove that the $\i$-category of qcc DQ-modules is a deformation of the category of quasi-coherent sheaves. Our integral representation theorem is a non-commutative analogue of To\"en's representation theorem.

To obtain  our integral representation theorem for DQ-modules, we first elucidate the linear structure of categories of DQ-modules. In particular, we show that the categories of qcc DQ-modules, introduced in \cite{dgaff}, and which is  the analogue of the derived category of quasi-coherent sheaves in the theory of DQ-modules, is enriched in the $\i$-categories of cohomologically complete $\Chbar$-modules, where $\Chbar:=\CC[[\hbar]]$ denotes the ring of formal power series in $\hbar$ with coefficients in $\CC$. This permit to describe precisely the action of $\Der(\Chbar)$ on the category of qcc modules. This allows us to establish that the category of qcc modules is a deformation of the category of quasicoherent sheaves in the sense of \cite{dagx} (see Theorem \ref{thm:defor} and Corollary \ref{cor:defor}) establishing a direct relations between the theory of DQ-modules and derived deformation theory (see also \cite{Pridham2018}). This result may also be relevant to better understand the relation between DQ-modules and the approach of deformation quantization introduced and developed in \cite{CPTVV}. 

Following the approach of \cite{mordg,BFN}, we obtain an equivalence between linear $\i$-functors between categories of qcc DQ-modules on algebraic varieties and qcc modules over the product of these varieties (see Theorem \ref{thm:FMDQ} and Formula \eqref{eq:FMformula}).
Moreover, we show that this equivalence associates to a DQ-module on the product of the varieties the corresponding Fourier-Mukai transforms. These results allow us to prove that an equivalence between $\i$-categories of qcc modules induces an equivalence between $\i$-categories of quasi-coherent sheaves (Corollary \ref{cor:semiequi}). 
We  also study the functors preserving compact objects. In the category of qcc modules, compact and dualizable objects do not coincide. The compact objects are the coherent DQ-modules which are of torsion with respects to the deformation parameters while the coherent DQ-modules are the dualizable one. We prove that the categories of functors preserving compact qcc modules on proper algebraic varieties is equivalent to the category of coherent DQ-modules on the product of the two algebraic varieties (see Theorem \ref{thm:Intrepcomp}). Hence, in our setting, we could summarize the situation by dualizable DQ-modules preserves the compact one. At first sight, the study of integral transforms for DQ-modules may appear similar to the one for quasi-coherent sheaves. They are several key differences which are for instance the nature of the linear structure encountered on the categories of the qcc sheaves, the fact that dualizable and compact objects are different in the categories of qcc modules (even though the underlying variety is smooth)  contrary to the situation studied in \cite{BFN}. Indeed, in our setting, the perfect objects are dualizable but not compact, the compact being the dualizable objects of torsion with respect to the deformation parameters.

In the non smooth setting, the study of integral transforms between categories of coherent objects (not necessarily perfect) is pursued in \cite{Ben-Zvi2013}.
For this purpose, they introduce a categorial notion of coherent object; however, it seems that this notion does not encompass the notion of coherent DQ-module.
In both situations, coherent objects are objects satisfying a weaker finiteness condition than that of compact objects, but otherwise the conditions are very different in nature.

It is known that DQ-modules encode certain aspects of Poisson geometry. For instance, if $X$ is a complex variety endowed with a DQ-algebroid $\A_X$, a coherent $\A_X$-module without torsion with respect to the deformation parameter has a coisotropic support for the Poisson structure associated to the DQ-algebroid $\A_X$. Here, the condition regarding the absence of torsion with respect to the deformation parameter --- usually called $\hbar$-torsion --- is of central importance. Indeed, the theory of DQ-modules takes place over $\Chbar$, which suggests one of two possible approaches: Either work over the central fiber and pass to the semi-classical limit (set the deformation parameter equal to zero); if this is done appropriately problems often reduce to  questions of complex geometry. Alternatively, work away from the central fiber, over $\CC((\hbar))$. In this situation, one study objects over $\A_X^\loc=\CC((\hbar)) \otimes_\Chbar \A_X$  which encodes information related to the Poisson structure induced by the DQ-algebroid. In this setting, it is not anymore possible to take the semi-classical limit has the deformation parameter has been inverted. Hence, the standard approach is to study modules over $\A_X^{\loc}$-module generated by $\A_X$-module satisfying some finiteness conditions (e.g. coherency). This approach is very similar to the one encountered in the study of $\D_X$-modules (in particular in the analytic setting where one often requires the $\D_X$-modules to be \textit{good}). This is why, as a first step, we study the case where $\hbar$ is not inverted. It is likely that the study of $\A_X^{\loc}$-modules will also require developing the theory of ind-coherent DQ-modules (see for instance \cite{BDMN2017}).

The present paper is organized as follows. We start by briefly reviewing, in Section \ref{sec:hcomp}, the notions of algebroid stacks and their derived $\i$-categories. We then move on to consider $\hbar$-completion, also called cohomological completion, and review the main properties of the $\gr$ functor.
In Section \ref{sec:DQmod}, we recall aspects of the theory of DQ-modules and define the $\i$-category of ``qcc'' (quasicoherent cohomologically complete) DQ-modules as well as the subcategory of coherent DQ-modules --- note that, though the $\i$-category of qcc-modules is not monoidal, the coherent objects neverthless correspond to dualizable objects in an appropriate sense. Finally, we prove that the category $\Der_\qcc(\A_X)$ of qcc-modules over a DQ-algebroid stack $\A_X$ is a deformation of the category of quasi-coherent sheaves $\Der_\qcoh(\O_X)$.
In Section \ref{sec:FMQ}, relying on the techniques of $\hbar$-completions of $\i$-categories of Section \ref{sec:hcomp}, we define operations for qcc-modules and in particular Fourier-Mukai transforms for $\i$-categories of qcc-modules and establish our integral representations theorems for DQ-modules. Finally, in an appendix, we collect some results concerning nilpotent, local, and complete $\i$-categories, first construct various $\i$-functors used in the theory of DQ-modules relying on techniques of model categories. In a second appendix, we construct all the operations needed for a robust theory of DQ-modules over an algebroid stack in the $\i$-categorical context.\\

\noindent \textbf{Acknowledgments:} Fran\c cois Petit would like to thank Benjamin Hennion, Mauro Porta and Marco Robalo for generously sharing their knowledge and for several useful scientific conversations. He is also grateful to Pierre Schapira, Bertrand To\"en and Damien Calaque for insightful conversations. He acknowledge the support of the Idex ``Universit\'e de Paris 2019".

\section{\texorpdfstring{$\hbar$}{h}-complete \texorpdfstring{$\infty$}{infinity}-categories} \label{sec:hcomp}

\subsection{Modules over algebroid stacks} \label{app:modelstruct}

We very briefly review the theory of algebroid stacks and their modules. For a detailed study of $k$-algebroid stacks and their operations, we refer to \cite{APMicro,APStack,KS3}.
We fix a topological space $X$, which in practice will be the underlying space of a smooth complex variety $(X,\O_X)$.

We let $k$ denote a commutative ring, and recall that a $k$-enriched category is a category enriched in the symmetric monoidal category of (ordinary) $k$-modules.
We write $\Cat_k$ for the category of $k$-enriched categories.

A $k$-enriched stack $\mathfrak{S}$ is a (pseudo)functor $\Open(X)^{\op}\to\Cat_k$ which satisfies the stack condition.
We write $\Shv_{\Cat_k}(X)\subset\Fun(\Open(X)^{\op},\Cat_k)$ of the category of $k$-enriched stacks on $X$.

More generally, we could consider instead a sheaf of commutative rings $\mathscr{R}$ on $X$. An $\mathscr{R}$-enriched stack is a $\ZZ$-enriched stack $\mathfrak{S}$ such that the sheaf $\mathscr{E}nd(\id_\mathfrak{S})$ of endomorphisms of the identity functor $\id_\mathfrak{S}\colon\mathfrak{S}\to\mathfrak{S}$ is a sheaf of commutative $\mathscr{R}$-algebras.

\begin{definition}
Let $X$ be a topological space.
\begin{enumerate}
\item
A $k$-enriched stack $\A$ on $X$ is locally nonempty if there exists a covering of $X$ by open subsets $\lbrace U_i\rbrace_{i \in I}$ such that for every $i \in I$, the category $\A(U_i)$ is non-empty.
\item 
A $k$-enriched stack $\A$ on $X$ is locally connected if, for any open subset $U$ of $X$ and any pair of objects $\alpha,\beta\in\A(U)$, there exists an open subset $j:V\subset U$ such that $j^*\alpha\cong j^*\beta$ in $\A(V)$.
\item
A $k$-algebroid stack $\A$ on $X$ is a $k$-enriched stack on $X$ which is locally nonempty and locally connected.
\end{enumerate}
\end{definition}

\begin{definition}
Let $X$ be a topological space and let $\R$ be a sheaf of commutative rings on $X$.
\begin{enumerate}
\item
An $\mathscr{R}$-algebroid stack $\A$ on $X$ is a $\ZZ$-algebroid stack endowed with a morphism of sheaves of rings $\mathscr{R} \to \mathscr{E}nd(\id_\A)$ (the sheaf of endomorphisms of the identity functor).
\item
An $\mathscr{R}$-algebroid $\A$ is said to be invertible if $\mathscr{R}|_U \to \mathscr{E} nd_{\A}(\sigma)$ is an isomorphism for any open subset $U$ of $X$ and any $\sigma \in \A(U)$.
\end{enumerate}
\end{definition}

\begin{example}
If $A$ is a $k$-algebra, we denote by $A^+$ the $k$-enriched category with one object having $A$ as the endomorphism ring of this object. Let $\A$ be a sheaf of algebras on $X$ and consider the prestack which associates to an open set $U$ the $k$-enriched category $\A(U)^+$. The stack associated to this prestack, denoted $\A^{\dagger}$, is an algebroid stack.
\end{example}

We write $k_X$ for the sheaf associated to the constant presheaf with value $k$ and $\sMod_{k_X}$ for the $k$-enriched stack of sheaves of $k$-modules on $X$. If $\mathfrak{S}_1$ and $\mathfrak{S}_2$ are two $k$-enriched stacks, we denote by $\Fct_k(\mathfrak{S}_1, \mathfrak{S}_2)$ the category of $k$-enriched morphisms of stacks from $\mathfrak{S}_1$ to $\mathfrak{S}_2$. We set
\begin{equation*}
\Mod_\A:=\Fct_k(\A,\sMod_{k_X}).
\end{equation*}
Finally, we write $\Mod_{\A}$ for the (ordinary) category of left $\A$-modules. It is a Grothendieck abelian category (see \cite{KS3} for details).

\subsection{Derived categories}\label{subsec:Derived}

Let $k$ be a commutative ring and let $\A$ be a sheaf of algebras, or an algebroid stack. Recall that $\Mod_{\A}$ is a Grothendieck abelian category which is tensored over $\Mod_{k}$, the abelian category of $k$-modules. It follows that the category of chain complexes $\Ch(\A)$ of objects of $\Mod_{\A}$ can be endowed with a proper combinatorial model structure (see for instance \cite[Proposition 1.3.53]{HA}) where the weak equivalences are the quasi-isomorphisms and a map of chain complexes $f \colon \M^\bullet \to \N^\bullet$ is a cofibration if for every $k \in \ZZ$, the induced map $f^k \colon \M^k \to \N^k$ is a monomorphism. This model structure on $\Ch(\A)$ is called the injective model structure. The $\i$-category presented by $\Ch(\A)$ endowed with this model structure is denoted $\Der(\A)$ and is called the derived category of $\Mod_{\A}$. It is the category obtained by localizing the $\i$-category $\mathrm{N}_{\mathrm{dg}}(\Ch(\A))$ with respects to the quasi-isomorphisms \cite[Proposition 1.3.5.13 \& 1.3.5.14]{HA}. Its homotopy category is the ``usual'' derived category of $\Mod_\A$ and we denote it $\hDer(\A)$. 
We will also use the semi-free model structure, subordinated to a covering $\U$ of $X$, on $\Ch(\A)$ (see Appendix \ref{sec:modstruct}). The $\i$-category underlying this model category is equivalent to $\Der(\A)$ by \cite[Lemma 1.3.4.21]{HA} and \cite[Proposition 1.3.4.15]{HA}

It follows from \cite[Proposition 1.3.5.9 \& Proposition 1.3.5.21]{HA} that $\Der(\A)$ is a presentable stable $\i$-category and from \cite[Example D.1.3.9]{SAG} that $\Der(\A)$ is a $k$-linear $\i$-category. This implies that $\Der(\A)$ has a $\Der(k)$-enriched mapping space functor
\begin{equation*}
\uMap_{\Der(\A)}( - \, , -) \colon \Der(\A)^{\op} \times \Der(\A) \to \Der(k).
\end{equation*}

Let $f \colon X \to Y$ be a map of topological spaces and 
let $\A$ be an algebroid stack on $Y$. Then there is an 
algebroid stack functorially associated with $\A$ and $f$ 
on $X$, denoted $f^{-1}\A$ (see \cite[\textsection 2.1]{KS3} for details), and we have the pair of adjoint functors 
\begin{equation*}
f^{-1} \colon \Mod_{\A} \rightleftarrows \Mod_{f^{-1} \A} \colon f_\ast.
\end{equation*}
They induce a Quillen adjunction
\begin{equation*}
f^{-1} \colon \Ch(\A) \rightleftarrows \Ch(f^{-1} \A) \colon f_\ast
\end{equation*}
with $f^{-1}$ exact and $k$-linear. Using \cite[Proposition 1.3.4.26]{HA}, one obtains a $k$-linear functor
\begin{equation*}
f^{-1} \colon \Der(\A) \rightarrow \Der(f^{-1} \A).
\end{equation*}
Its right adjoint is obtained as the right derived functor of $f_\ast$ by using the injective model structure on $\Ch(f^{-1} \A)$. Hence, we obtain a functor in $\Prr[\mathrm{st}]$
\begin{equation*}
f_\ast \colon \Der(f^{-1} \A) \rightarrow \Der(\A).
\end{equation*}

\begin{proposition}\label{prop:commlim}
Assume that $X$ and $Y$ are Noetherian topological spaces of finite Krull dimensions. The functor $f_\ast:\D(f^{-1} \A) \to \D(\A)$ commutes with colimits and is $k$-linear.
\end{proposition}

\begin{proof} Since the $\i$-category $\D(f^{-1} \A)$ is stable we just need to check that $f_\ast$ commutes with coproducts.
Let $(\M_i)_{i \in I}$ be a family of objects of $\D(f^{-1} \A)$. To check that the morphism
\[
\alpha:\bigoplus_{i \in I} f_\ast \M_i \to f_\ast( \bigoplus_{i \in I} \M_i)
\]
is an equivalence, we observe that the question is local on $X$ and thus reduces to question in ordinary sheaf theory. Since $X$ is a Noetherian topological space of finite Krull dimension, this follows from \cite[3.9.3.2]{Lip}.\footnote{Although the result in \cite{Lip} is stated for $\O_X$-modules, the spaces considered here are Noetherian topological spaces of finite Krull dimensions and thus its proof adapts word for word to our setting.}
Finally, $k$-linearity follows from the fact that $f_\ast$ commutes with colimits and \cite[Example D.1.3.9]{SAG}.
\end{proof}

Similarly, the functor $\fHom(-,-) \colon \Mod_{\A}^{\op} \times \Mod_{\A} \to \Mod_{\CC^\hbar_X}$
is left exact. Using the injective model structure on $\Mod_{\A}$, we obtain a right derived functor
\begin{equation*}
\fMap_{\A_X}(-,-) \colon \Der(\A)^{\op} \times \Der(\A) \to \Der(\CC^\hbar_X).
\end{equation*}
Note that $\Gamma(X,\fMap_{\A}(-,-)) \simeq \uMap_{\A} (-,-)$.

\subsection{\texorpdfstring{$\hbar$}{h}-completion}

The introduction of the notion of derived completion in the framework of DQ-modules was inspired by the following homological formulation of completeness with respect to the $\hbar$-adic topology. It was originally formulated by Kashiwara and Schapira in order to deduce the Grauert direct image theorem for DQ-modules from the usual version for $\O_X$-modules.

\begin{lemma}[{\cite[Lemma 1.5.4]{KS3}}]
Let $\R$ be a sheaf of $\ZZ[\hbar]$-algebras without $\hbar$-torsion and set $\R^\loc=\ZZ[\hbar,\hbar^{-1}] \otimes_{\ZZ[\hbar]} \R$. Let $\M$ be an $\R$-module and assume that $\M$ has no $\hbar$-torsion.
\begin{enumerate}
\item $\widehat{\M} \simeq \fExt^1_{\R}(\R^{\loc} / \R, \M)$, where $\widehat{\M}$ is the completion of $\M$ with respect to the $\hbar$-adic topology.
\item The module $\M$ is $\hbar$-adically complete if and only if $\fExt^j_{\R}(\R^{\loc}, \M)=0$ for $j=0,1$.
\end{enumerate}
\end{lemma}

Since one of our main motivations is the study of DQ-modules, we will mostly be interested in the case where the base ring is $\Chbar$, the ring of formal power series with coefficients in $\CC$. However, most everything remains true in the slightly more general case where $A$ is a discrete $\ZZ[\hbar]$-algebra without $\hbar$-torsion.

Recall that if $\C\in\Mod_{\Der(A)}$ is a $\Der(A)$-linear $\i$-category, we say that:
\begin{enumerate}
		\item[(1)] An object $N \in \C$ is $\hbar$-nilpotent if the localization $A[\hbar^{-1}] \otimes_A N $ vanishes.
		\item[(2)] An object $L \in \C$ is $\hbar$-local if the canonical map $L \to A[\hbar^{-1}] \otimes_A L$ is an equivalence.
		\item[(3)] An object $M \in \C$ is $\hbar$-cohomologically complete if $M^{A[\hbar^{-1}]} \simeq 0$.
\end{enumerate}

We denote by $\C^\nil$, $\C^\loc$, $\C^\cc$ the full subcategories of $\C$ respectively spanned by the nilpotent, local and complete objects. We write
\begin{align*}
i_\lor \colon \C^\nil \subset \C && j_\ast \colon \C^\loc \subset \C && i_\land \colon \C^\cc \subset \C
\end{align*}
for the respective fully faithful inclusion, $i^\lor$ for the right adjoint of $i_\lor$, and $j^\ast$ and $i^\land$ for the left adjoints of $j_\ast$ and $i_\land$, respectively. We refer the reader to Appendix \ref{app:nilloccomp} and in particular to Proposition \ref{prop:fundccnil} for more details.

We write $\Der := \Der(A)$, $\Der_{\cc}:=\Der_{\cc}(A)$ and $\Der_{\loc}:=\Der_{\loc}(A)$ for the full subcategories of $\hbar$-complete and $\hbar$-local objects, respectively. It is standard (see for instance \cite[Theorem 7.1.2.1]{HA} and \cite[Remark 7.1.2.3]{HA} for a proof) that $\Der_{\loc}$ is equivalent to $\Der(A[\hbar^{-1}])$.
The semiorthogonal decomposition $(\Der_\loc,\Der_\cc)$ of $\D$ implies that
\[
\D_\loc\overset{j_*}{\to}\D\overset{i^\land}{\to}\D_\cc
\]
is a Verdier sequence in $\PrL$.
Moreover, the fact that the ideal $I=(\hbar)\subset A$ is finitely generated (even principal) implies that $j_*$ admits a left adjoint $j^*:\D\to\D_{\loc}$.

Since $\D_\cc$ is a commutative $\D$-algebra via the $\hbar$-completion functor $i^\land:\D\to\D_\cc$, we obtain a ``restriction along $i^\land$'' functor
\[
i_*\colon\Mod_{\D_\cc}\to\Mod_\D
\]
which admits a left adjoint ``base change along $i^\land$'' functor
\[
i^*:\Mod_{\D}\to\Mod_{\D_\cc}.
\]
The fact that $\D_\cc$ is an idempotent (hence commutative) $\D$-algebra implies, using \cite[Propositions 4.8.2.9 \& 4.8.2.10]{HA}, that the forgetful functor $i_*:\Mod_{\D_\cc}\to\Mod_\D$ is fully faithful with essential image those $\D$-linear $\i$-categories $\M$ of the form $\M\simeq\D_\cc\otimes_\D\N$ for some $\D$-linear $\i$-category $\N$.
In fact we can just take $\N=\M$, since $\M\to i_*i^*\M$ is an equivalence if $\M$ is in the image of $i_*:\Mod_{\D_\cc}\to\Mod_{\D}$.

\begin{proposition}
The $\i$-category $\Der_{\cc}$ admits the structure of a presentably symmetric monoidal $\i$-category under $\Der(A)$.
That is, there exists a presentably symmetric monoidal structure on $\Der_\cc$ such that the functor $\comp \colon \Der \to \Der_{\cc}$ is symmetric monoidal and the tensor product functor $\Der_{\cc}\times\Der_{\cc}\to\Der_{\cc}$ is given on objects by the formula $(M,N) \mapsto \comp( i_\land(M) \otimes_A i_\land(N))$.
\end{proposition}

\begin{proposition}
Any $\D$-linear $\i$-category $\M$ admits a semiorthogonal decomposition of the form $(\M^{\loc},\M^{\cc})$.
\end{proposition}
\begin{proof}
Cotensoring with the exact triangle $i_\lor i^\lor A\to A\to j_*j^*A$ yields an exact triangle of endofunctors
$
(-)^{j_*j^*A}\to\id\to(-)^{i_\lor i^\lor A}\colon\M\to\M.
$
\end{proof}

\begin{proposition}
The canonical map $\D_{\loc}\otimes_\D\M\to\M$ factors through the full subcategory $\M^{\loc}\subset\M$ of local objects.
\end{proposition}

\begin{proof}
The functor may be identified with the map
\[
\Funr_{\D}(\M^{\op},\D_{\loc})\to\Funr_{\D}(\M^{\op},\D)\simeq\M
\]
which composes a right adjoint $\D$-linear functor $F:\D_{\loc}^{\op}\to\M$ with the right adjoint $\D$-linear functor $\D_{\loc}\to\D$.
Note that the equivalence $\M\to\Funr_\D(\M^{\op},\D)$ is adjoint to the $\D$-enriched mapping object functor $\underline{\Map}_\M(-,-):\M^{\op}\times\M\to\D$, and sends the object $M$ to the $\D$-linear right adjoint functor $\underline{\Map}_\M(-,M):\M^{\op}\to\D$ (this functor evidently preserves limits and $\kappa$-filtered colimits for some $\kappa$, hence admits a left adjoint); the fact that this is an equivalence is the assertion that any such $\D$-linear functor is representable.
It therefore suffices to show that the object $M$ representing the $\D$-linear right adjoint functor $\underline{\Map}_\M(-,M):\M^{\op}\to\D$ is $\hbar$-local whenever this functor factors through the full subcategory $\D_\loc\subset\D$ of $\hbar$-local objects.

To see this, it suffices to check that the canonical map $M^{j_*j^*A}\to M$ is an equivalence, which is to say that $M^{i_\vee i^\vee A}\simeq 0$.
Indeed, for any object $L$ of $\M$,
\begin{align*}
\underline{\Map}_\M(L,M^{i_\vee i^\vee A}) & \simeq\underline{\Map}_\M(i_\vee i^\vee A\otimes_A L,M)
\simeq\underline{\Map}_\M(L,M)^{i_\vee i^\vee A}
\simeq 0
\end{align*}
since $\underline{\Map}_\M(L,M)$ is $\hbar$-local by hypothesis.
\end{proof}

\begin{proposition}\label{prop:comparecompl}
For any $\D$-linear $\i$-category $\M$, there is a commutative diagram of Verdier sequences
\[
\xymatrix{
\D_{\loc}\otimes_\D \M\ar[r]\ar[d] & \D\otimes_\D \M\ar[r]\ar[d] & \D_{\cc}\otimes_\D\M\ar[d]\\
\M^{\loc}\ar[r] & \M\ar[r] & \M^{\cc}}
\]
in which the vertical maps are equivalences.
\end{proposition}

\begin{proof}
The vertical map in the middle is the canonical equivalence $\D\otimes_\D\M\to\M$, and the vertical map on the left is the factorization $\D_\loc\otimes_\D\M\to\M$ through the full subcategory $\M^{\loc}\subset\M$ of $\hbar$-local objects.
The vertical map on the right is the functor induced by taking horizontal cofibers; equivalently, the commutative square on the right (or on the left) may be rewritten in its right adjoint form as the commutative square
\[
\xymatrix{
\Funr_\D(\M^{\op},\D) & \Funr_\D(\M^{\op},\D_\cc)\ar[l]\\
\M\ar[u] & \M^\cc\ar[u]\ar[l]}
\]
in which the vertical maps send the object $M$ of $\M$ to the functor $\underline{\Map}_\M(-,M):\M^{\op}\to\D$ represented by $M$ (observe that if $M$ is cohomologically complete then the representable functor $\underline{\Map}_\M(-,M):\M^{\op}\to\D_\cc\subset\D$ factors through the full subcategory of cohomologically complete objects).
In particular, all the functors in this square are fully faithful and the left vertical functor is the canonical equivalence $\M\simeq\Funr_\D(\M^{\op},\D)$.
In fact, this square is a pullback: if $F:\M^{\op}\to\D_\cc$ is a $\D$-linear right adjoint functor such that $F\simeq\underline{\Map}_\M(-,M)$ as a functor $\M^{\op}\to\D$, then $M$ is cohomologically complete.
Indeed, just as above, for any object $L$ of $\M$,
\begin{align*}
\underline{\Map}_\M(L,M^{j_* j^* A}) & \simeq\underline{\Map}_\M(j_* j^* A\otimes_A L,M)
\simeq\underline{\Map}_\M(L,M)^{j_* j_* A}
\simeq 0
\end{align*}
since $\underline{\Map}_\M(L,M)$ is $\hbar$-complete.
We therefore conclude that the map $\D_\cc\otimes_\D\M\to\M^{\cc}$ is an equivalence, and by an analogous argument (or by general facts about semiorthogonal decompositions) we conclude that the map $\D_\loc\otimes_\D\M\to\M^{\loc}$ is an equivalence. 
\end{proof}

Let $\Mod_\D^{\cc}\subset\Mod_\D$ denote the full subcategory of $\Mod_\D$ spanned by the $\D$-linear $\i$-categories which are cohomologically complete.
The previous discussion can be summarized as follows.

\begin{corollary}
The forgetful functor $\Mod_{\D_\cc}\to\Mod_\D$ induces an equivalence $\Mod_{\D_\cc}\simeq\Mod_\D^\cc$.
\end{corollary}

\begin{corollary}
Let $f^* \colon \M \to \N$ be a morphism in $\Mod_\D$. The diagram
\begin{equation*}
\xymatrix{
\D_\cc \otimes_\D \M\ar[r]^-{\id_{\D_\cc} \otimes f^*}\ar[d] & \D_\cc \otimes_\D \N \ar[d]\\
\M^{\cc}\ar[r]^-{i^\land \circ f^* \circ \, i_\land} & \N^{\cc},}
\end{equation*}
where the vertical maps are the equivalences of Proposition \ref{prop:comparecompl}, commutes up to natural equivalence.
Moreover, the composite $i^\land\circ f^*\circ i_\land:\M^\cc\to\N^\cc$ is $\D_\cc$-linear.
\end{corollary}

\begin{proof}
Supposing the square commutes, it follows that $i^\land\circ f^*\circ i_\land$ is the basechange along the commutative algebra map $i^\land:\D\to\D_\cc$ of the $\D$-linear functor $f^*$, and therefore is $\D_\cc$-linear.
To see that the square commutes, writing the tensor products as $\i$-categories of $\D$-linear right adjoint functors $\D_\cc\otimes_\D\M\simeq\Funr_\D(\M^{\op},\D_\cc)$ and $\D_\cc\otimes_\D\N\simeq\Funr_\D(\N^{\op},\D_\cc)$, we check that the associated diagram of right adjoints commutes.
The counit of the adjunction induces a natural transformation of functors
\[
\underline{\Map}_\N(-,f_*(-))\to
\underline{\Map}_\M(f^*(-),-)\colon \M^{\op}\times\N^{\cc}\to\D_\cc
\]
in which the source and target functors represent the two ways of traversing the associated diagram of right adjoints.
The square commutes because the composite transformation is an equivalence.
\end{proof}

\begin{lemma}
We set $T := (A[\hbar^{-1}]/A)[-1]$. Then $i_\land i^\land \simeq ( - )^{T}$.
\end{lemma}

\begin{proof}
Since $A$ has no $\hbar$-torsion, the canonical map $A \to A[\hbar^{-1}]$ is injective. Thus, we get the following exact sequence
\begin{equation*}
0 \to A \to A[\hbar^{-1}] \to A[\hbar^{-1}] / A \to 0. 
\end{equation*}
Now using Proposition \ref{prop:fundccnil} \ref{item:funccnil}, we obtain an equivalence $i_\lor i^\lor A \simeq T$. Finally, it follows from Proposition \ref{prop:compformula} that $i_\land i^\land \simeq ( - )^{T}$.
\end{proof}
\begin{definition}
Let $\C\in\Mod_{\Der(A)}$ and let $M \in \C$. The object $M$ is of uniform $\hbar$-torsion if there exist a positive integer $n$ such that the map $\hbar^n\colon M \to M$ is null.
\end{definition}

\begin{lemma}\label{lem:unicomp}
Let $\C\in \Mod_{\Der(A)}$. An object of $\C$ of uniform $\hbar$-torsion is $\hbar$-complete.
\end{lemma}

\begin{proof}
Let $M \in \C$ and assume it is of uniform $\hbar$-torsion. Then there exists positive integer $n$ such that the map 
\begin{equation}\label{mor:zeroh}
M \stackrel{\hbar^n}{\to} M 
\end{equation}
is a zero map. Cotensoring the morphism \eqref{mor:zeroh} by $A[\hbar^{-1}]$, we obtain the morphism
\begin{equation}\label{mor:cozeroh}
M^{A[\hbar^{-1}]} \stackrel{\hbar^n}{\to} M^{A[\hbar^{-1}]},
\end{equation} 
which is a zero map.
This morphism \eqref{mor:cozeroh} can also be obtained from the isomorphism $A[\hbar^{-1}] \stackrel{\hbar^n}{\to}A[\hbar^{-1}]$ by the bifunctoriality of the cotensorization. Indeed, applying the functor $M^{(-)}$, we obtain an isomorphism $M^{A[\hbar^{-1}]} \stackrel{\hbar^n}{\to} M^{A[\hbar^{-1}]}$.
This implies that $M^{A[\hbar^{-1}]}\simeq 0$.
\end{proof}

\begin{lemma}\label{lem:unitor}
Let $\C \in \Mod_{\Der(A)}$ and $M \in \C$ an object of uniform $\hbar$-torsion. Then $T \tensor M \simeq M$.
\end{lemma}

\begin{proof}
Tensor the exact triangle $T \to A \to A[{\hbar}^{-1}]$ by $M$ and use that $A[\hbar^{-1}] \otimes M \simeq 0$ since $M$ is of uniform $\hbar$-torsion.
\end{proof}

\begin{proposition}\label{prop:preservcomp}
Let $\C\in\Mod_{\Der(A)}$. Let $M$ be a compact object of $\C$ of uniform $\hbar$-torsion. Then $i^\land(M)$ is a compact object of $\C^{\cc}$.
\end{proposition}

\begin{proof}. 
Let $\colim_j N_j$ be a filtered colimit in $\C^{\cc}$. Then
\begin{align*}
\Map_{\C^{\cc}}(i^{\land}(M),\colim_j N_j) &\simeq \Map_{\C}(M,i_\land i^\land(\underset{j}{\colim}\, i_\land(N_j)))\\
                                  &\simeq \Map_{\C}(T \tensor M,\underset{j}{\colim}\, i_\land(N_j))\\
                                  &\simeq \Map_{\C}(M,\underset{j}{\colim}\, i_\land(N_j)) \quad \textnormal{(Lemma \ref{lem:unitor})}\\
                                  &\simeq \underset{j}{\colim} \Map_{\C}( M,i_\land(N_j)) \quad \textnormal{(compacity of $M$)}\\
                                  &\simeq \underset{j}{\colim} \Map_{\C^{\cc}}( i^\land (M), \, N_j).           
\end{align*}
\end{proof}
\begin{proposition}\label{prop:compcompgen}
Let $\C$  be an $A$-linear $\i$-category which is compactly generated by a compact generator of uniform $\hbar$-torsion. Then the $\hbar$-completion functor $i^{\land} \colon \C \to \C^\cc$ is an equivalence.
\end{proposition}

\begin{proof}
Let $G$ be a compact generator of $\C$ of uniform $\hbar$-torsion. Since the $\hbar$-completion functor is a left adjoint, $\comp{(G)}$ is a generator of $\C^\cc$, which is compact by Proposition \ref{prop:preservcomp}. Moreover, the canonical map
\begin{equation*}
\uMap_\C(G,G) \stackrel{i^\land}{\longrightarrow} \uMap_{\C^\cc}(i^\land(G),i^\land(G))
\end{equation*}
is an isomorphism by Lemma \ref{lem:unicomp}. Since $i^\land$ commutes with colimits, it follows that it is an essentially surjective and fully faithful $A$-linear functor.
\end{proof}

\subsection{The \texorpdfstring{$\gr$}{gr} functor}

Let $\R$ be a sheaf of $\ZZ[\hbar]$-algebras which is flat over $\ZZ[\hbar]$ and such that $\R_0:=\R/\hbar\R$ is a commutative $\ZZ$-algebra, and we view $\R_0$ as an $\R_0\otimes_{\ZZ[\hbar]}\R^{\op}$-module.
We write $p:\R\to\R_0$ for the projection, which is a map of sheaves of $\ZZ[\hbar]$-algebras.
The basechange functor is a right exact functor $\Mod_{\R}\to\Mod_{\R_0}$, which we can derive to an exact functor on the level of derived $\i$-categories.
Alternatively, the basechange functor is left adjoint to the forgetful functor $\iota:\Der(\R)\to\Der(\R_0)$ which is exact on the level of abelian categories.

\begin{definition}
We define a sequence of adjoint functors
\begin{equation}
\xymatrix{\Der(\R_0) \ar[r]^-{\iota} & \Der(\R) \ar@/_1.5pc/[l]_-{\gr} \ar@/^1.5pc/[l]^-{\cogr}
}
\end{equation}
as follows:
let $\iota:\Der(\R_0)\to\Der(\R)$ denote the restriction functor.
Then $\iota$ is monadic and comonadic with left and right adjoints $\gr:\Der(\R)\to\Der(\R_0)$ and $\cogr:\Der(\R)\to\Der(\R_0)$, respectively.
\end{definition}
\begin{remark}
The functor $\gr:\Der(\R)\to\Der(\R_0)$ is the functor $\M \mapsto \R_0\otimes_\R \M$ induced by tensoring with the $\R_0\otimes_{\ZZ[\hbar]}\R^{\op}$-module $\R_0$, and the functor $\cogr:\Der(\R)\to\Der(\R_0)$ is the functor
$
\M \mapsto \fMap_\R(\R_0, \M)
$ 
induced by cotensoring with the $\R\otimes_{\ZZ[\hbar]}\R_0^\op$-module $\R_0$. 
\end{remark}

\begin{remark}
If $\R$ is a $\CC^\hbar$-algebra, then we induce the same functor $\Der(\R)\to\Der(\R_0)$ if we view $\R_0$ as an $\R_0\otimes_{\ZZ[\hbar]}\R^{\op}$-module or as an $\R_0\otimes_{\CC^\hbar}\R^{\op}$-module.
\end{remark}

We have the following basic relation between $\gr$ and $\cogr$.

\begin{proposition}
Let $\M \in \Der(\R)$. Then $\iota \circ \gr \M \simeq \iota \circ \cogr \circ \M [1]$.
\end{proposition}

\begin{proof}
When $\R_0$ is endowed with its natural structure of $\R_0\otimes_{\ZZ[\hbar]}\R^{\op}$-module we denote it by $\R_{0R}$, when endowed with its natural structure of $\R\otimes_{\ZZ[\hbar]}\R_0^\op$-module we denote it by $_R\R_{0}$, and when endowed with its structure of $\R \otimes_{\ZZ[\hbar]}\R^{\op}$-module we denote it by $_R\R_{0 R}$. Then,
\begin{align*}
\iota \circ \gr \M &\simeq \; _R \R_0 \otimes_{\R_0} \R_{0 R} \otimes_\R \M\\
                   &\simeq _R \R_{0 R} \otimes_\R \M \simeq (\R \stackrel{\hbar}{\to} \R) \otimes_\R \M \\
                   &\simeq \fMap_\R(\R \stackrel{\hbar}{\to} \R, \M[1])\simeq \fMap_\R(_R \R_{0 R},\M[1])\\
                   &\simeq \iota\circ \cogr \circ \M[1].
\end{align*}
Here, $\R \stackrel{\hbar}{\to} \R$ is the complex with $\R$ in degree $-1$ and $0$, zero otherwise and the multiplication by $\hbar$ as differential in degree zero.
\end{proof}

\begin{corollary}
Let $\M \in \Der(\R)$. Then, for every $i \in \ZZ$,
$
\Hn^i(\gr \M) \simeq \Hn^{i+1}(\cogr \M).
$
\end{corollary}

\begin{remark}
There is a more precise relation between $\gr$ and $\cogr$ in the specific case of DQ-algebroid stacks (see \cite[Proposition 2.3.6]{KS3}). 
\end{remark}

Recall the following result from \cite[Corollary 1.5.9]{KS3}.

\begin{proposition}
The restriction of the functor $\gr$ to $\Der_\cc(\R)$ is conservative.
\end{proposition}

\begin{proposition}\label{prop:isogr}
We have the following formal properties of the $\gr$ functor:
\begin{enumerate}
\item
Let $\M\in\Der(\R^{\op})$ and $\N\in\Der(\R)$.
Then
$
\gr(\M)\otimes_{\R_0}\gr(\N)\simeq\gr(\M\otimes_\R\N).
$
\item
Let $\M\in\Der(\R)$ and $\N\in\Der(\R)$.
Then
$
\gr(\fMap_\R(\M,\N))\simeq\fMap_{\R_0}(\gr(\M),\gr(\N)).
$
\end{enumerate}
\end{proposition}

\begin{example}
We consider $\Chbar$ the algebra of formal power series with coefficient in $\CC$ and $\CC \simeq \Chbar / \hbar \Chbar$. The categories $\Der(\Chbar)$ and $\Der(\CC)$ are presentably symmetric monoidal categories and the functor
$\gr \colon \Der(\Chbar) \to \Der(\CC)$ given by $M \mapsto \CC \otimes_\Chbar M$ refines to a symmetric monoidal functor.
\end{example}
\begin{proposition}[{\cite[Prop. 3.8]{dgaff}}]\label{prop:gr_cc}
The natural transformation $\gr \circ \id \to \gr \circ \, i_\land \circ \comp{}$
induced by the unit transformation $\id\to i_\land i^\land$ is an equivalence in $\Der(\R_0)$.
\end{proposition}

\begin{remark}
These results are readily extended to the case of a $\ZZ[\hbar]$-algebroid stacks.
\end{remark}

\section{DQ-algebras and DQ-modules}\label{sec:DQmod}

\subsection{DQ-modules}

In this subsection, we review some standard notions concerning DQ-algebras and refer the reader to \cite{KS3} for a more detailed study.
In what follows, $(X,\O_X)$ denotes a smooth complex algebraic variety.
Just as $\Chbar=\varprojlim_{n\in\NN} \Chbar / \hbar^n \Chbar$, we define the sheaf of $\Chbar$-algebras analogously, as the limit
\begin{equation*}
\O_X^\hbar:=\varprojlim_{n \in \NN} \O_X \tensor_\CC (\Chbar / \hbar^n \Chbar).
\end{equation*}

\begin{definition}
A star-product denoted $\star$ on $\O_X^\hbar$ is a $\Chbar$-bilinear associative law satisfying
\begin{equation*}
f \star g = \sum_{i \geq 0} P_i(f,g) \hbar^i \;\; \textnormal{for every} \;f, \; g \in \O_X,
\end{equation*}
where the $P_i$ are bi-differential operators such that for every $f, g \in \O_X, P_0(f,g)=fg$ and 
$P_i(1,f)=P_i(f,1)=0$ for $i>0$. The pair $(\O_X^\hbar, \star)$ is called a star-algebra. 
\end{definition}

\begin{example}
Let $U$ be an open subset of $T^\ast\AA^{n}$ endowed with a symplectic coordinate system $(x;u)$ with $x=(x_1,\ldots,x_n)$ and  $u=(u_1,\ldots,u_n)$.
There is  a star-algebra $(\O_{U} ^\hbar, \star)$ on $U$ given by
\begin{equation*}
f \star g = \sum_{\alpha \in \NN^n} \dfrac{\hbar^{|\alpha|}}{\alpha !} (\partial^\alpha_u f) (\partial^\alpha_x g).
\end{equation*}
\end{example}

\begin{definition}
A DQ-algebra $\A_X$ on $X$ is a $\Chbar$-algebra which is locally isomorphic, as a $\Chbar$-algebra, to a star-algebra.
\end{definition}

\begin{remark} It follows immediately from the above definition that
\begin{enumerate}
\item $\hbar$ is central in $\A_X$,
\item $\A_X$ has no $\hbar$-torsion,
\item $\A_X$ is $\hbar$-complete,
\item $\A_X / \hbar \A_X \simeq \O_X$.
\end{enumerate} 
\end{remark}

\begin{definition}
A DQ-algebroid stack on $X$ is a $\Chbar$-algebroid stack such that, for each open subset $U$ of $X$ and each $\sigma \in \A_X(U)$, the $\Chbar$-algebra $\fEnd(\sigma)$ is a DQ-algebra on $U$. 
\end{definition}

We typically refer to DQ-algebroid stacks as DQ-algebroids.
\begin{definition}
Let $\A_X$ be a DQ-algebroid.
A DQ-module on $(X,\A_X)$ is a left module over $\A_X$.
We write $\Mod_{\A_X}$ for the abelian category of left $\A_X$-modules.
\end{definition}
\begin{lemma}[{\cite[Lemma 2.2.8]{KS3}}]
If $\A_X$ is a DQ-algebroid then the opposite algebroid, denoted $\A_{X^\mathrm{a}}$, is a DQ-algebroid.
In particular, if $\A_X$ is a DQ-algebra then so is $\A_{X^\mathrm{a}}$.
\end{lemma}

Recall that if $X$ and $Y$ are two  smooth complex algebraic varieties endowed with DQ-algebroid $\A_X$ and $\A_Y$ then $X \times Y$ is canonically endowed with a DQ-algebroid (see \cite[\textsection 2.3]{KS3} for the  definition of the operation $\ubtimes$ for DQ-algebroids)
\begin{equation*}
\A_{X \times Y}:= \A_X \ubtimes \A_Y.
\end{equation*}
 Moreover, if $\A_X$, $\A_Y$ are DQ-algebras and $p\colon X \times Y \to X$ denotes the projection onto $X$ then there is a canonical morphism of algebras
\begin{equation*}
p^{-1} \A_X \to \A_{X \times Y}
\end{equation*}
and this morphism is flat.
By symmetry, a similar result holds for the projection $q \colon X \times Y \to Y$.

We refer the reader to \cite[\textsection 2.3]{KS3} for details concerning the algebroid $\gr \A_X$. Here, we rapidly summarize a few key points.
Given a DQ-algebroid $\A_X$, there is a $\CC$-algebroid denoted $\gr \A_X$. It is an invertible $\O_X$-algebroid and is equipped with a canonical map of $\CC$-algebroids $\A_X \to \gr \A_X$. This algebroid induces a $\gr \A_X \otimes_{\CC^{\hbar}_X} \A_{X^\mathrm{a}}$-module that we still denote $ \gr \A_X$. As an $\A_X \otimes \A_{X^\mathrm{a}}$-module, $\gr \A_X$ is isomorphic to $\CC \otimes_{\Chbar} \A_X$. 
This bi-module gives rise to a functor $\gr \colon \Mod_{\A_X} \to \Mod_{\gr \A_X}$ given informally by the rule
\begin{equation}\label{def:grDQ}
\M \mapsto \gr \A_X \otimes_{\A_X} \M.
\end{equation}
Moreover, the functor $\gr$ is left adjoint to the exact functor 
$\iota\colon \Mod_{\gr \A_X} \to \Mod_{\A_X}$ induced by the canonical morphism $\A_X \to \gr \A_X$.
On an algebraic variety, the algebroid stack $\gr \A_X$ is equivalent to the stackification of the structure sheaf $\O_X$ (see \cite[Remark 2.1.17]{KS3}). This induces an equivalence between $\Mod_{\O_X}$ and $\Mod_{\gr \A_X}$.

Let $\delta \colon X \hookrightarrow X \times X$ be the diagonal embedding. Recall that the canonical $\A_{X} \otimes \A_{X^\mathrm{a}}$-bimodule $\A_X$ is a bi-invertible $\A_{X} \otimes \A_{X^\mathrm{a}}$-module (see \cite{KS3} Definition 2.1.10 for the notion of bi-invertibility). It follows from \cite[Corollary 2.4.4]{KS3} that $\delta_\ast \A_X$ is a coherent $\A_{X \times X^\mathrm{a}}$-module simple along the diagonal. We denote it $\A_\Delta$.

We write $\omega_X$ the dualizing complex for $\A_X$. It is a bi-invertible $\A_X$-module. We refer to \cite[\textsection 3.3]{KS3} for a detailed treatment of duality for DQ-modules.

\begin{remark}
	In \cite{KS3}, duality theory for DQ-modules is established in the analytic setting. Though, we are working in the algebraic setting this is not a problem for us since we can either use the GAGA theorem for DQ-modules \cite{Chen2010} as we only use duality theory when the variety is proper. Note that in the algebraic setting, it is also possible to reproduce the results of \cite[\textsection 3.3]{KS3} relying directly on Grothendieck duality for algebraic varieties and qcc modules.
\end{remark}

\subsection{The \texorpdfstring{$\i$}{infinity}-category of qcc-modules}

From now on, $(X,\O_X)$ is a smooth complex algebraic variety, equipped with the Zariski topology, and endowed with a DQ-algebroid $\A_X$. 
We define the category $\Der_\qcc(\A_X)$ and collect some of its properties.

The category $\Der(\A_X)$ is an object of $\Mod_{\Der(\Chbar)}$. The categories $\Der(\gr \A_X)$ is a $\CC$-linear category and $\Der_\qcoh(\gr \A_X)$ is a full $\CC$-linear subcategory of $\Der(\gr \A_X)$. Thus they are  $\Chbar$-linear categories through restriction of scalars via the presentably symmetric monoidal functor $\gr \colon \Der(\Chbar) \to \Der(\CC)$. Moreover, the functor $\gr \colon \Der(\A_X) \to \Der(\gr \A_X)$ given by $\M \mapsto \gr \A_X \otimes_{\A_X} \M$ is a morphism in $\Mod_{\Der(\Chbar)}$ (see sub-section \ref{subsec:gr} for more details).
The $\i$-category $\Der_{\gqcoh}(\A_X)$ of graded quasicoherent $\A_X$-modules is defined as the pullback
\begin{equation}\label{eq:square}
\xymatrix{
\Der_\gqcoh(\A_X)\ar[r]\ar[d] & \Der(\A_X)\ar[d]^-{\gr}\\
\Der_\qcoh(\gr\A_X)\ar@{^{(}->}[r] & \Der(\gr\A_X)}
\end{equation}
in $\Mod_{\Der(\Chbar)}$.
In particular, we view $\Der_\gqcoh(\A_X)$ as an object of $\Mod_{\Der(\Chbar)}$. 

\begin{remark}\label{rem:limcomputation} Recall that $\PrL$ has all small limit and colimits and that the inclusion of $\PrL$ into $\Cat_\i$ commutes with limits. Similarly, limits in $\Mod_{\D(\Chbar)}$ are computed in $\PrL$, and hence in $\Cat_\i$.
\end{remark}
The $\i$-category $\Der_\gqcoh(\A_X)$ is the full $\Chbar$-linear sub-category of $\Der(\A_X)$ spanned by the objects $\M$ such that $\gr \M \in \Der_\qcoh(\gr\A_X)$. 

\begin{remark}
One also obtains an $\i$-category equivalent to $\Der_\gqcoh(\A_X)$ working with $\cogr$ instead of $\gr$.
However, it is technically easier to work with the $\gr$ functor instead of the $\cogr$ functor, as the former behaves well with respect to tensor products.
\end{remark}

\begin{remark}
An object $\M$ is cohomologically complete in $\Der_\gqcoh(\A_X)$ if and only if it is cohomologically complete in $\Der(\A_X)$. This follows from the fact that the $\Chbar$-linear structure on $\Der_\gqcoh(\A_X)$ is obtain by restricting to it the $\Chbar$-linear structure of $\Der(\A_X)$.
\end{remark}

\begin{definition}
The $\i$-category $\Der_{\qcc}(\A_X)$ of graded quasicoherent cohomologically complete $\A_X$-modules is the cohomological completion of $\Der_{\gqcoh}(\A_X)$, viewed as an object of $\Mod_{\Der_{\cc}(\Chbar)}$.	
\end{definition}
We refer to graded quasicoherent cohomologically complete $\A_X$-modules simply as qcc-modules.

\begin{remark} It follows from the definition of $\Der_\qcc(\A_X)$ that it is a full $\D_\cc(\Chbar)$-linear full subcategory of $\Der_\cc(\A_X)$.
\end{remark}

First, we remark that the category $\Der(\gr \A_X)$ is $\hbar$-nilpotent. Indeed, let $\F \in \Der(\gr \A_X)$. Then,
\begin{align*}
\CC((\hbar)) \otimes_{\Chbar} \F &\simeq \gr (\CC((\hbar))) \otimes_\CC \F\simeq 0.
\end{align*}
Hence, it follows from Proposition \ref{prop:carcatcc} that $\Der(\gr \A_X)$ is cohomologically complete. We make similar observations for $\Der_\qcoh(\gr \A_X)$. We now consider the completion of the functor
\begin{equation*}
\gr \colon \Der (\A_X) \to \Der(\gr \A_X).
\end{equation*}
Using that $\Der(\gr \A_X)$ is cohomologically complete, we obtain a $\Der_\cc(\Chbar)$-linear functor
\begin{equation*}
\widehat{\gr} \colon \Der_\cc (\A_X) \to \Der(\gr \A_X)
\end{equation*}
by restricting $\gr$ to $\Der_\cc (\A_X)$ (hence we write $\gr$ instead of $\widehat{\gr}$).
Finally, we complete the cartesian square \eqref{eq:square} in $\Mod_{\Der(\Chbar)}$
%
in order to obtain a commutative diagram of the form
\begin{equation}\label{diag:qccpullback}
\xymatrix{
\Der_\qcc(\A_X)\ar[r]\ar[d] & \Der_\cc(\A_X)\ar[d]^-{\gr}\\
\Der_\qcoh(\gr\A_X)\ar@{^{(}->}[r] & \Der(\gr\A_X)}
\end{equation}
in  $\Mod_{\Der_\cc(\Chbar)}$.

\begin{proposition}
The diagram \eqref{diag:qccpullback} is a pullback in $\Mod_{\Der_\cc(\Chbar)}$.
\end{proposition}

\begin{proof}
This follows from the definition and Remark \ref{rem:limcomputation}.
\end{proof}

We finally collect a few facts concerning qcc-modules. 

\begin{lemma}[{\cite[Corollary 3.14]{dgaff}}]
If $\N \in \Der_\qcoh(\gr \A_X)$ then $\iota(\N) \in \Der_\qcc(\A_X)$.
\end{lemma}

\begin{proposition}[{\cite[Prop 2.17]{dgaff}}]\label{prop:presercomgen}
The functors
\[ \xymatrix{
\Der_\qcc(\A_X) \ar@<.4ex>[r]^-{\gr} & \Der_\qcoh(\gr \A_X) \ar@<.4ex>[l]^-{\iota}
}
\]
preserve compact generators
\end{proposition}

\begin{corollary}[{\cite[Corollary 3.19]{dgaff}}]
The category $\Der_\qcc(\A_X)$ is compactly generated by a single compact generator.
\end{corollary}

\subsection{The \texorpdfstring{$\i$}{infinity}-category of coherent DQ-modules}

Let $(X, \O_X)$ be a smooth complex algebraic variety endowed with an algebroid stack $\A_X$.

We denote by $\Derb_\coh(\A_X)$ the full subcategory of $\Der(\A_X)$ spanned by the objects $\M$ such that for $|n|>>0$, $\Hn^n(\M)=0$ and for every $n \in \ZZ$, $\Hn^n(\M) \in \Mod_\coh(\A_X)$ is the abelian category of coherent $\A_X$-module.

\begin{proposition}
	The category $\Derb_\coh(\A_X)$ is a full stable subcategory of $\Der_\qcc(\A_X)$.
\end{proposition}

\begin{proof}
	The fact $\Derb_\coh(\A_X)$ is a full subcategory of $\Der_\qcc(\A_X)$ is a direct consequence of \cite[Theorem 1.6.1]{KS3}. Since  $\Der_\qcc(\A_X)$ is stable, it follows from \cite[Lemma 1.1.3.3]{HA} that we only need to check that $\Derb_\coh(\A_X)$ is stable by translation (which is obvious) and by cofibers. Let $f\colon \M \to \N$ a morphism of $\Derb_\coh(\A_X)$. Then, we have the following cofiber sequence in $\Der_\qcc(\A_X)$
	\begin{equation*}
	\M \stackrel{f}{\to} \N \to \F
	\end{equation*}
where $\F$ is the cofiber of $f$. This induces a long exact sequence
\begin{equation*}
\cdots \to \Hn^i(\M) \stackrel{\Hn^i(f)}{\longrightarrow} \Hn^i(\N) \to \Hn^i(\F) \to \Hn^{i+1}(\M) \to \cdots
\end{equation*}	
Hence, for every $i \in \ZZ$, we get the exact sequence
\begin{equation*}
0 \to \coker(\Hn^{i}(f)) \to \Hn^i(\F) \to \ker(\Hn^{i+1}(f) \to 0.
\end{equation*}
Since $\coker(\Hn^{i}(f))$ and $\ker(\Hn^{i+1}(f)$ are coherent so is  $\Hn^i(\F)$.
\end{proof}

\begin{proposition}
	The $\i$-category $\Derb_\coh(\A_X)$ is an idempotent complete stable $\i$-category.
\end{proposition}

\begin{proof}
Since  $\Derb_\coh(\A_X)$ is stable it is sufficient, by \cite[Lemma 1.2.4.6]{HA}, to check that its homotopy category $\hDerb_\coh(\A_X)$ is idempotent complete to prove that the $\i$-category is idempotent complete. The triangulated category $\hDer(\A_X)$ has countable coproduct hence by \cite[Proposition 1.6.8]{Neetri}, it is idempotent complete. It follows that idempotent in $\hDerb_\coh(\A_X)$ splits in $\hDer(\A_X)$. Let $e \colon \M \to \M$ be an idempotent of $\hDerb_\coh(\A_X)$ and $f\colon \M \to \N$ and $g \colon \N \to \M$ a splitting of $e$ in $\hDer(\A_X)$. Let us check that $\N\in\hDerb_\coh(\A_X)$. Since $f \colon \M \to \N$ is a strict epimorphism (in the sense that it admits a section) it follows that for every $n \in \mathbb{Z}$, $\Hn^n(f) \colon \Hn^n(\M) \to \Hn^n(\N)$ is a strict epimorphism. Hence, $\Hn^n(\N)$ is a locally finitely generated $\A_X$-module and in particular for $|n|$ sufficiently large $\Hn^n(\N)$ is zero. Moreover, since $g: \N \to \M$ is a strict monomorphism, for every $n \in \mathbb{Z}$, $\Hn^n(\N)$ is a locally finitely generated submodule of $\Hn^n(\M)$. Thus, it is coherent. This proves that $\N \in \hDerb_\coh(\A_X)$.
\end{proof}

It follows from the above proposition that $\Derb_\coh(\A_X)$ is an object of $\Cat_{\i,\mathrm{idem}}^\mathrm{ex}$. Moreover, $\Der_\qcc(\A_X)$ is a $\Der(\CC^\hbar)$-module and the subcategory $\Derb_\coh(\A_X)$ of $\Der_\qcc(\A_X)$ is stable by tensorization by object of $\Perf(\Chbar)$. Hence, we get the following lemma.

\begin{lemma}\label{lem:linstruccoh}
The $\Der(\Chbar)$-module structure on $\Der_\qcc(\A_X)$ induces a $\Perf(\Chbar)$-module structure on $\Derb_\coh(\A_X)$; i.e., $\Derb_\coh(\A_X) \in \Mod_{\Perf(\CC^\hbar)}(\Cat_{\i,\mathrm{idem}}^\mathrm{ex})$.
\end{lemma}

\begin{definition}
	Let $\M \in \Der_\qcc(\A_X)$. The module $\M$ is dualizable if the canonical morphism
	\begin{equation*}
	\psi \colon \fMap_{\A_X}(\M,\A_X) \tensor_{\A_X} (-) \to \fMap_{\A_X}(\M,-)
	\end{equation*}
	is an equivalence of functors $\Der(\A_X)\to\Der(\CC_X^\hbar)$.
\end{definition}

\begin{lemma}\label{lem:grdual}
	Let $\M \in \Der_\qcc(\A_X)$. If $\M$ is dualizable in $\Der_\qcc(\A_X)$ then $\gr \M$ is dualizable in $\Der_\qcoh(\O_X)$.
\end{lemma}

\begin{proof}
	Let $\M \in \Der_\qcc(\A_X)$ be dualizable, so the canonical morphism
	$
	\fMap_{\A_X}(\M,\A_X) \otimes_{\A_X} \M \to \fMap_{\A_X}(\M,\M) 		
	$
	is an equivalence. Applying the $\gr$ functor and using Proposition \ref{prop:isogr}, we obtain the equivalence
	$
	\psi_{\gr\M} \colon \fMap_{\O_X}(\gr \M,\O_X) \otimes_{\O_X} \gr \M \stackrel{\sim}{\to} \fMap_{\O_X}(\gr \M, \gr \M). 		
	$
	As $\Der_\qcoh(\O_X)$ is a closed symmetric monoidal $\i$-category, this implies that $\gr \M$ is dualizable.
\end{proof}

\begin{lemma}\label{lem:dualcoh}
	Let $\M \in \Der_\qcc(\A_X)$. If $\gr \M$ is dualizable in $\Der_\qcoh(\O_X)$ then $\M \in \Derb_\coh(\A_X)$.
\end{lemma}

\begin{proof}
	Since $X$ is a smooth algebraic variety and $\gr \M$ is dualizable, $\gr \M \in \Derb_\coh(\O_X)$. Indeed, by \cite[Proposition 3.6]{BFN}, on an algebraic variety perfect and dualizable objects of $\Der_\qcoh(\O_X)$ are the same. Moreover, the smootheness assumption implies that locally any coherent sheaf has a free resolution of finite length which in turns implies that the objects of $\Derb_\coh(\O_X)$ are perfect. Since $\M \in \Der_\qcc(\A_X)$, it is cohomollogically complete. Hence, it follows from \cite[Proposition 3.11]{FMQ} that $\M \in \Der^b_\qcc(\A_X)$. Now applying \cite[Theorem 1.6.4]{KS3}, we obtain that $\M \in \Derb_\coh(\A_X)$.
\end{proof}

\begin{proposition}
	Let $\M \in \Der_\qcc(\A_X)$. The module $\M$ is dualizable if and only if $\M \in \Derb_\coh(\A_X)$.
\end{proposition}

\begin{proof}
	\noindent (i) Assume $\M \in \Derb_\coh(\A_X)$. The result follows from \cite[Theorem 1.4.8]{KS3}.\\
	\noindent (ii) Assume that $\M$ is dualizable in $\Der_\qcc(\A_X)$. Hence, by Lemma \ref{lem:grdual}, $\gr \M$ is dualizable in $\Der_\qcoh(\O_X)$. Now, applying Lemma \ref{lem:dualcoh} it follows that $\M \in \Derb_\coh(\A_X)$.
\end{proof}

\begin{theorem}[{\cite[Theorem 3.20]{dgaff}}]\label{thm:compactobj}
	 The compact objects of $\Der_\qcc(\A_X)$ are the $\M$ such that $\M \in \Derb_\coh(\A_X)$ and $\A_X^\loc \tensor_{\A_X} \M=0$ (with $\M$ considered as an object of $\Der(\A_X)$) .
\end{theorem} 

\begin{remark}
The preceding results show that, while every compact object of $\Der_\qcc(\A_X)$ is dualizable, not every dualizable object of $\Der_\qcc(\A_X)$ is compact.
For instance, even over the point, $\CC^\hbar$, viewed as complex concentrated in degree zero, is dualizable is not compact, since $\CC((\hbar))\neq 0$.
\end{remark}

\begin{lemma}
Let $\M \in \Derb_\coh(\A_X) \subset \Der(\A_X)$. Then $\A_X^\loc \otimes_{\A_X} \M\simeq 0$ if and only if there exists $n >0$ such that $\hbar^n \M\simeq 0$.
\end{lemma}

\begin{proof}
Since $\M \in \Derb_\coh(\A_X)$, $\Hn^i(\M) = 0$ but for finitely many $i$. Hence, we can assume that $\M \in \Mod_\coh(\A_X)$.
As $\M$ is coherent, it is locally finitely generated. As $X$ is an algebraic variety, we can find a finite open cover by affine opens $(U_i)_{1 \leq i \leq n}$ such that $\M|_{U_i}$ is finitely generated. Thus, we can assume that $X$ is affine and that $\M$ is finitely generated on $X$ by finitely many sections $s_1,\ldots,s_p$. As $X$ is a Noetherian topological space, we have an isomorphism $\CC^{\hbar,\loc} \otimes \Gamma(X;\M) \simeq \Gamma(X; \A_X^{\loc} \otimes_{\A_X} \M)$.
Moreover, $\Gamma(X; \A_X^{\loc} \otimes_{\A_X} \M)=0$. It follows that there is an $n >0 $ such that $\hbar^n s_j =0$.
The converse is clear: if $\hbar^n\M=0$ for some $n>0$ then $\A^{\loc}_X\otimes_{\A_X}\M\simeq 0$.
\end{proof}

\subsection{\texorpdfstring{$\D_{\qcc}$}{Dqcc} as a deformation of \texorpdfstring{$\D_{\qcoh}$}{Dqcoh}}

We establish that $\D_{\qcc}(\A_X)$ is a deformation of the presentable stable $\i$-category $\D_{\qcoh}(\gr\A_X)$.

\begin{definition}
A formal deformation of a $\C$-linear $\i$-category $\C$ is a pair $(\C^\hbar,\mu)$ where $\C^\hbar$ is an object of $\Mod_{\Der(\Chbar)}$ and $\mu$ is a $\Der(\CC)$-linear equivalence
\begin{equation*}
\mu \colon \Der(\CC) \otimes_{\Der(\Chbar)} \C^\hbar \to \C.
\end{equation*}
\end{definition}

See \cite[Section 5.3]{dagx} for details. We will require the following Lemma.

\begin{lemma}\label{lem:gentensproduct}
Let $A$ be a commutative ring and suppose given  stable $A$-linear $\i$-categories $\M$ and $\N$ which are generated by compact objects $G_\M$ and $G_\N$.
\begin{enumerate}
\item $\M \otimes_{\Der(A)} \N$ is generated by the compact object $G_\M \otimes_A G_\N$.
\item For $M, \; M^\prime \in \C^\omega$ and $N, \; N^\prime \in \D^\omega$, there is a canonical isomorphism
\begin{equation*}
\uMap_{\M \otimes_{\Der(A)} \N}(M \otimes_A N, M^\prime \otimes_A N^\prime) \simeq \uMap_{\M}(M, M^\prime) \otimes_A \uMap_{\N}(N, N^\prime).
\end{equation*}
\end{enumerate}
\end{lemma}

\begin{proof}
The proof is similar to the one of \cite[Lemma 3.2]{AGH}
\end{proof}

The morphism of commutative algebra objects $\gr \colon \D(\CC^\hbar) \to \D(\CC)$ induces a basechange functor $\Mod_{\D(\CC^\hbar)}\to\Mod_{\D(\CC)}$ which is left adjoint to the forgetful functor $\Mod_{\D(\CC)}\to\Mod_{\D(\CC^\hbar)}$.
By construction of $\D_{\qcc}(\A_X)$, we have a $\CC^\hbar$-linear functor
$
\gr \colon \D_\qcc(\A_X) \to \D_\qcoh(\gr \A_X).
$
By adjunction, this induces a morphism in $\Mod_{\D(\CC)}$
\begin{equation}\label{map:equidefor}
\Psi \colon \D(\CC) \otimes_{\D(\CC^\hbar)}\D_\qcc(\A_X) \to \D_\qcoh(\gr \A_X).
\end{equation}

\begin{theorem}\label{thm:defor}
The morphism \eqref{map:equidefor} is an equivalence in $\Mod_{\D(\CC)}$.
\end{theorem}

\begin{proof}
Let $G$ be a compact generator of $\D_\qcoh(\gr \A_X)$ which exists by \cite[Theorem 3.1.1]{BVdB}. It follows from Proposition \ref{prop:presercomgen} that $\iota(G)$ is a compact generator of $\D_\qcc(\A_X)$. Hence by Lemma \ref{lem:gentensproduct}, $\D(\CC) \otimes_{\D(\CC^\hbar)}\D_\qcc(\A_X)$ is compactly generated by the object $ \CC \otimes_{\D(\CC^\hbar)} \iota(G)$. For brevity, we set $\Gr\Der_\qcc(\A_X) :=\D(\CC) \otimes_{\D(\CC^\hbar)}\Der_\qcc(\A_X)$

Unraveling the construction of the morphism \eqref{map:equidefor}, we obtain 
\[
\xymatrix{
\D(\CC) \otimes_{\D(\CC^\hbar)}\D_\qcc(\A_X) \ar[r]^-{\id \otimes \gr} &  \D(\CC) \otimes_{\D(\CC^\hbar)} U(\D_\qcoh(\gr \A_X)) \ar[r] & \D_\qcoh(\gr \A_X).
}
\]
This implies that 
\[
\Psi (\CC \otimes_{\D(\CC^\hbar)} \iota(G))\simeq \CC \otimes \gr \iota(G) \simeq \gr \iota (G).
\]
It follows from Proposition \ref{prop:presercomgen} that $\gr \iota (G)$ is a compact generator of $\D_\qcoh(\gr \A_X)$.
To check that $\Psi$ is an equivalence, it is sufficient to show that
\[
\Map_{\Gr \Der_\qcc(\A_X)}(\CC \tensor_{\D(\CC^\hbar)}\iota(G),\CC \tensor_{\D(\CC^\hbar)}\iota(G)) \to \Map_{\D_\qcoh(\gr \A_X)}(\Psi(\CC \tensor_{\D(\CC^\hbar)}\iota(G)),\Psi(\CC \tensor_{\D(\CC^\hbar)}\iota(G)))
\]
is an isomorphism in $\D(\CC)$.
.

This a consequence of the following commutative diagram

\[
\xymatrix@C=0.5pc{\uMap_{\Gr\Der_\qcc(\A_X)}(\CC \! \tensor_{\Der(\CC^\hbar)} \! \iota(G),\CC \! \tensor_{\Der(\CC^\hbar)} \! \iota(G)) \ar[r]& \uMap_{\Der_\qcoh(\gr \A_X)}(\Psi(\CC \! \tensor_{\Der(\CC^\hbar)} \! \iota(G)),\Psi(\CC \! \tensor_{\Der(\CC^\hbar)} \! \iota(G)))\\
\uMap_{\D(\CC)}(\CC,\CC)\tensor_{\CC^\hbar} \uMap_{\D_\qcc(\A_X)}(\iota(G),\iota(G)) \ar[r] \ar[u]^-{\wr} & \uMap_{\D_\qcoh(\gr \A_X)}(\gr \iota(G)),\gr \iota(G)) \ar[u]_-{\wr}\\
\CC\tensor_{\CC^\hbar} \uMap_{\D_\qcc(\A_X)}(\iota(G),\iota(G)) \ar[u]^-{\wr}\ar[r]^-{\sim}& \uMap_{\D_\qcoh(\gr \A_X)}(\gr \iota(G),\gr \iota(G)) \ar@{=}[u]\\
}
\]
where the upper left vertical map is an isomorphism by Lemma \ref{lem:gentensproduct}. The lowest horizontal map of the diagram is an isomorphism by Proposition \ref{prop:isogr}.
\end{proof}

\begin{corollary}\label{cor:defor}
The pair $(\D_{\qcc}(\A_X),\Psi)$ is a deformation of $\D_{\qcoh}(\gr \A_X)$.
\end{corollary}

\section{Fourier-Mukai transforms for DQ-modules}\label{sec:FMQ}

In this section we develop the theory of integral transforms for DQ-modules at the $\i$-categorical level. This allows us to take into account the linear structures naturally present on $\i$-categories of DQ-modules. Taking into consideration these linear structures and relying on techniques of $\hbar$-completions allow us to set up the study of integral representation theorems for DQ-modules in the framework of the Morita theory of linear categories between categories of DQ-modules. Afterwards, we are able to obtain several integral representation theorems.

\subsection{Operations on qcc-modules}

Let $f:X_1\to X_2$ be a morphism of smooth complex algebraic varieties and assume $X_2$ is endowed with a DQ-algebroid stack $\A_2$. By Subsection \ref{subsec:Derived} we have adjoint functors $f^{-1} \dashv f_\ast$
\[
\xymatrix{
\D(f^{-1}\A_2)  \ar@<.5ex>[r]^-{f_\ast} & \D(\A_2)\ar@<.5ex>[l]^-{f^{-1}}.
}
\]
such that  $f_\ast$ preserves cohomological completeness by Proposition 1.5.12 of \cite{KS3}.
Thus $f_\ast$ induces a functor
\[
f_\ast:\D_\cc(f^{-1} \A_2) \to \D_\cc(\A_2).
\]
The left adjoint of $f_\ast$ is the cohomological completion
\[
\hat{f}^{-1}:\D_\cc(f^{-1} \A_2) \to \D_\cc(\A_2)
\]
of $f^{-1}$, which we will typically denote $f^{-1}$ instead of $\hat{f}^{-1}$.

The relative tensor prodcuts of DQ-modules is the $\Chbar$-linear functor 
\begin{align*}
- \utensor_{\A_2} - &\colon \Der(\A_{12^\mathrm{a}}) \times \Der(\A_{23^\mathrm{a}}) \to \Der(p_{13}^{-1}\A_{13^\mathrm{a}}), \quad (\K_1, \K_2)  \mapsto p_{12}^{-1} \K_1 \tensor_{p_{12}^{-1}\A_{12^\mathrm{a}}} \A_{123} \tensor_{p_{23^\mathrm{a}}^{-1} \A_{23}} p_{23}^{-1} \K_2
\end{align*}
(see Appendix \ref{subsec:tensors} for details). Its cohomological completion induces the functor
\begin{align*}
- \cutensor_{\A_2} - &\colon \D_\cc(\A_{12^\mathrm{a}}) \times \D_\cc(\A_{23^\mathrm{a}}) \to \D_\cc(p_{13}^{-1}\A_{13^\mathrm{a}}).
\end{align*}
\begin{proposition} Let $\K_1 \in \D_\cc(\A_{12^\mathrm{a}})$ and $\K_2 \in \D_\cc(\A_{23^\mathrm{a}})$. There are canonical isomorphisms
\begin{align*}
\hat{p}_{12}^{-1} \K_1 \ctensor_{p_{12}^{-1} \A_{12^\mathrm{a}}} (\A_{123} \ctensor_{p_{23^\mathrm{a}}^{-1} \A_{23}} \hat{p}_{23}^{-1} \K_2 )\stackrel{\sim}{\longleftarrow} 
\K_1 \cutensor_{\A_2} \K_2  \stackrel{\sim}{\longrightarrow} (\hat{p}_{12}^{-1} \K_1 \ctensor_{p_{12}^{-1}\A_{12^\mathrm{a}}} \A_{123}) \ctensor_{p_{23^\mathrm{a}}^{-1} \A_{23}} \hat{p}_{23}^{-1} \K_2.
\end{align*}
\end{proposition}
\begin{proof}
The unit of the adjunction between the completion and inclusion functors provide a map
\begin{equation*}
\K_1 \utensor_{\A_2} \K_2 \longrightarrow (\hat{p}_{12}^{-1} \K_1 \ctensor_{p_{12}^{-1}\A_{12^\mathrm{a}}} \A_{123}) \ctensor_{p_{23^\mathrm{a}}^{-1} \A_{23}} \hat{p}_{23}^{-1} \K_2.
\end{equation*}
Applying the completion functor to the above morphism, we get a map
\begin{equation}\label{mor:assocomtens}
\K_1 \cutensor_{\A_2} \K_2  \stackrel{\sim}{\longrightarrow} (\hat{p}_{12}^{-1} \K_1 \ctensor_{p_{12}^{-1}\A_{12^\mathrm{a}}} \A_{123}) \ctensor_{p_{23^\mathrm{a}}^{-1} \A_{23}} \hat{p}_{23}^{-1} \K_2.
\end{equation}
Applying the $\gr$ functor, using Proposition \ref{prop:gr_cc} and the conservativity of the $\gr$ functor, we get that the morphism \eqref{mor:assocomtens} is an isomorphism. The other isomorphism is proved similarly.
\end{proof}
The composition of cohomologically complete algebraic DQ-kernels is defined as follows.
\begin{definition}\label{def:DQconv}
We write $- \underset{2}{\circ} -\colon \D_\cc(\A_{12^\mathrm{a}}) \times \D_\cc(\A_{23^\mathrm{a}}) \to \D_\cc(\A_{13^\mathrm{a}})$ for the pushforward of the relative tensor product functor  $(\K_1, \K_2) \mapsto p_{13 \ast} (\K_1 \cutensor_{\A_2} \K_2)$.
\end{definition}

\begin{lemma}\label{lem:gr_comp} Let $\K_i \in \D_{\qcc}(\A_{i(i+1)^\mathrm{a}})$ $(i=1, 2)$. Then
$	\gr (\K_1 \underset{2}{\circ} \K_2)  \simeq \gr \K_1 \underset{2}{\circ} \gr \K_2.$
\end{lemma}

\begin{proof}
 By \cite[Proposition 1.4.4]{KS3}, the functor $\gr$ commutes with pushforward. Hence,
\begin{align*}
\gr (\K_1 \underset{2}{\circ} \K_2) & \simeq p_{13 \ast}(\gr(\K_1 \cutensor_{\A_2} \K_2))\\
& \simeq p_{13 \ast}(\gr(\K_1 \utensor_{\A_2} \K_2)) \quad \textnormal{(Prop. \ref{prop:gr_cc})}\\
& \simeq p_{13 \ast}(p_{12}^\ast \gr \K_1 \tensor_{\O_{123}} p_{23}^\ast\gr \K_2),
\end{align*}
where the final equivalence follows from Proposition \ref{prop:gr_cc} and \cite[Proposition 1.4.4]{KS3}.
\end{proof}

\begin{lemma}
Let $\K_i \in \D_{\qcc}(\A_{i(i+1)^\mathrm{a}})$ $(i=1, 2)$. The kernel $
\K_1 \underset{2}{\circ} \K_2$ is an object of $\D_{\qcc}(\A_{13^\mathrm{a}})$.
\end{lemma}

\begin{proof}
By definition of the composition of kernels, we know that $\K_1 \underset{2}{\circ} \K_2 \in \D_\cc(\A_{13^\mathrm{a}})$. It remains to check that $\gr (\K_1 \underset{2}{\circ} \K_2) \in \Der_\qcoh(\O_{13})$. By Lemma \ref{lem:gr_comp},
\begin{equation*}
\gr (\K_1 \underset{2}{\circ} \K_2)  \simeq \gr \K_1 \underset{2}{\circ} \gr \K_2.
\end{equation*}
and the composition of quasi-coherent kernels is again quasi-coherent.
\end{proof}
The above lemma implies that $- \underset{2}{\circ} -$ induces a functor $- \underset{2}{\circ} - \colon \D_\qcc(\A_{12^\mathrm{a}}) \times \D_\qcc(\A_{23^\mathrm{a}}) \to \D_\qcc(\A_{13^\mathrm{a}})$.
If $\K \in \D_{\qcc}(\A_{12^\mathrm{a}})$, this implies that the following functor (\ref{MukaiDQqcc}) is well-defined:
\begin{equation} \label{MukaiDQqcc}
\scalebox{0.96}{$\Phi_\K: \D_{\qcc}(\A_2) \to \D_{\qcc}(\A_1), \quad  \M \mapsto \K \underset{2}{\circ} \M = p_{1\ast}(\K \cutensor_{p_2^{-1}\A_2} \hat{p}_2^{-1} \M).$} 
\end{equation}

\begin{proposition}
Let $X_i$ $(i=1,2,3,4)$ be smooth algebraic varieties endowed with DQ-algebroid stacks $\A_i$  $(i=1,2,3,4)$ and let $\K_i \in \D_{\qcc}(\A_{i(i+1)^\mathrm{a}})$ $(i=1,2,3)$. There is a canonical isomorphism in $\D_\qcc(\A_{14^\mathrm{a}})$
\[
(\K_1 \underset{2}{\circ} \K_2) \underset{3}{\circ} \K_3 \simeq \K_1 \underset{2}{\circ} (\K_2 \underset{3}{\circ} \K_3).
\]
\end{proposition}

\begin{proof}
The proof is similar to the one of \cite[Propostion 3.2.4]{KS3}.
\end{proof}

\subsection{Finiteness results for integral tranforms}
In this short subsection, we recall some finiteness results that we will use subsequently.
We begin with the following result, which is a special case of Theorem 3.2.1 of \cite{KS3}.
\begin{theorem}\label{thm:prescoh}
Let $X_i$ $(i=1,2,3)$ be a smooth complex variety. For $i=1,2$ consider the product  $X_i \times X_{i+1}$ and let $\K_i \in \Derb_\coh(\A_{i(i+1)^\mathrm{a}})$. Assume that $X_2$ is proper. Then the object  $\K_1 \underset{2}{\circ}\K_2$ belongs to $\Derb_\coh(\A_{13^\mathrm{a}})$.
\end{theorem}

We also have the following result.
\begin{theorem}[{\cite[Theorem 3.12]{FMQ}}]\label{thm:compactkernel}
Let $X_1$ (resp. $X_2$) be a smooth complex projective algebraic variety endowed with a DQ-algebroid $\A_1$ (resp. $\A_2$). Let $\K \in \D_{\qcc}(\A_{12^\mathrm{a}})$. Assume that the functor $\Phi_\K: \D_{\qcc}(\A_2) \to \D_{\qcc}(\A_1)$ preserves compact objects. Then, $\K$ belongs to $\Derb_{\coh}(\A_{12^\mathrm{a}})$.
\end{theorem}

The situation where the kernel of the integral transform is coherent is particularly interessting  because of the following result which is a DQ-module analogues of well-known results of algebraic geometry. 

Let $X_1$ and $X_2$ be smooth complex projective algebraic varieties and let $\K \in \Derb_{\coh}(\A_{12^\mathrm{a}})$, we set
\begin{align*}
\K_R = \fMap_{\A_{12^a}}(\K,\A_{12^a}) \underset{2}{\circ} \omega_{2^a} && \K_L = \omega_{1^a} \underset{1^a}{\circ} \fMap_{\A_{12^a}}(\K,\A_{12^a})  
\end{align*}

\begin{proposition}[{\cite[Proposition 3.15]{FMQ}}]
 Let $\Phi_\K \colon \Derb_{\coh}(\A_{2}) \to \Derb_{\coh}(\A_{1})$ be the integral transform associated with $\K$ and $\Phi_{\K_R} \colon \Derb_{\coh}(\A_{1}) \to \Derb_{\coh}(\A_{2})$ (resp. $\Phi_{\K_L}$) the Fourier-Mukai functor associated with $\K_R$ (resp. $\K_L$). Then $\Phi_{\K_R}$ (resp. $\Phi_{\K_L}$) is right (resp. left) adjoint to $\Phi_\K$.
\end{proposition}

\begin{lemma}\label{lem:vanloc}
Let $p:Y\to X$ be a morphism of smooth complex algebraic varieties and suppose given a DQ-algebroid $\A$ on $X$. Let $\M \in \Der(p^{-1}\A)$. Then
\begin{equation*}
\CC_X^{\hbar,\loc} \otimes_{\CC_X^\hbar} p_\ast \M \simeq p_\ast( \CC_Y^{\hbar,\loc} \otimes_{\CC_Y^\hbar} \M).
\end{equation*}
\end{lemma}

\begin{proof}
Using that $\CC^{\hbar,\loc}$ is a filtered colimit the result follows immediatly from Proposition \ref{prop:commlim}.
\end{proof}

\begin{proposition}\label{prop:cohprecomp}
Let $X_1$ (resp. $X_2$) be a smooth complex proper algebraic variety endowed with a DQ-algebroid $\A_1$ (resp. $\A_2$). Let $\K \in \Derb_{\coh}(\A_{12^\mathrm{a}})$. Then the functor $\Phi_\K: \D_{\qcc}(\A_2) \to \D_{\qcc}(\A_1)$ preserves compact objects.
\end{proposition}

\begin{proof}
This follows directly from Theorem \ref{thm:prescoh}, Lemma \ref{lem:vanloc} and the description of compact objects of the qcc given in Theorem \ref{thm:compactobj}.
\end{proof}

\subsection{Symmetric monoidality of \texorpdfstring{$\D_\qcc$}{Dqcc}}

Recall that if $X$ and $Y$ are smooth complex algebraic varieties endowed with DQ-algebroid stacks $\A_X$ and $\A_Y$ then the variety $X \times Y$ is canonically endowed with a DQ-algebroid stack $\A_{X \times Y}$. Following \cite[p.68]{KS3}, we define a functor $\ubtimes \colon \D(\A_X) \times \D(\A_Y) \to \D(\A_{X \times Y})$ by the rule
\[
(\M , \N) \mapsto \A_{X \times Y} \tensor_{\A_X \btimes \A_Y}(\M \btimes_{\Chbar} \N).
\]
This induces a functor
\begin{equation}\label{funct:externalproduct}
\ubtimes \colon \D_{\gqcoh}(\A_X) \times \D_{\gqcoh}(\A_Y) \to \D_{\gqcoh}(\A_{X \times Y}).
\end{equation}
Indeed, let $\M$ and $\N$ be objects of $\D_{\gqcoh}(\A_X)$ and $\D_{\gqcoh}(\A_Y)$, respectively.
Then $\M \ubtimes \N$ belongs to $\D(\A_{X \times Y})$ and

\begin{align*}
\gr (\M \ubtimes \N) &\simeq \gr(\A_{X \times Y}) \tensor_{\A_{X \times Y}}  \A_{X \times Y} \tensor_{\A_X \btimes \A_Y}(\M \btimes_{\Chbar} \N)\\
                    & \simeq \gr \A_X \btimes \gr \A_Y \tensor_{\A_X \btimes \A_Y}(\M \btimes_{\Chbar} \N)\\
                    &\simeq \gr \M \btimes \gr \N .
\end{align*} 
(Here we have used the fact that $\gr(\A_{X \times Y})\cong\gr \A_X \btimes \gr \A_Y$).
It follows that $\M \ubtimes \N$ belongs to $\D_{\gqcoh}(\A_{X \times Y})$.
We have the following commutative diagram
\[
\xymatrix{\D(\A_X) \times \D(\A_Y)  \ar[r]^-{\ubtimes}  & \D(\A_{X \times Y})\\
            \D_\gqcoh(\A_X) \times \D_\gqcoh(\A_Y) \ar@{^{(}->}[u] \ar[r]^-{\ubtimes} & \D_\gqcoh(\A_{X \times Y}). \ar@{^{(}->}[u]
}
\]
Applying the cohomological completion functor to this diagram we obtain a commutative diagram
\[
\xymatrix{\D_\cc(\A_X) \times \D_\cc(\A_Y))  \ar[r]^-{\cubtimes}  & \D_\cc(\A_{X \times Y})\\
            \D_\qcc(\A_X) \times \D_\qcc(\A_Y) \ar@{^{(}->}[u] \ar[r]^-{\cubtimes} & \D_\qcc(\A_{X \times Y}), \ar@{^{(}->}[u]
}
\]
which by $\D_\cc$-linearity induces a commutative diagram
\[
\xymatrix{\D_\cc(\A_X) \tensor_{\D_\cc(\Chbar)} \D_\cc(\A_Y) \ar[r]^-{\cubtimes}  & \D_\cc(\A_{X \times Y})\\
            \D_\qcc(\A_X) \tensor_{\D_\cc(\Chbar)} \D_\qcc(\A_Y) \ar[u] \ar[r]^-{\cubtimes} & \D_\qcc(\A_{X \times Y}). \ar@{^{(}->}[u]
}
\]

\begin{lemma}\label{lem:compactness}
The functor $\cubtimes \colon \Der_\qcc(\A_X) \times \Der_\qcc(\A_Y)) \to \Der_\qcc(\A_{X \times Y})$ preserves compact objects.
\end{lemma}

\begin{proof}
Let $\M$ (resp. $\N$) be a compact object of $\D_\qcc(\A_X)$ (resp. $\D_\qcc(\A_Y)$). Notice that $\M \ubtimes \N \in \Derb_\coh(\A_{X \times Y})$ and is cohomologically complete by Lemma \ref{lem:unicomp}. Consequently
$
\M \cubtimes \N \simeq \M \ubtimes \N
$
and $\A_{X \times Y}^{\loc} \tensor_{\A_{X \times Y}} \M \cubtimes \N\simeq 0$. Theorem \ref{thm:compactobj} implies that $\M \cubtimes \N$ is a compact object of $\D_\qcc(\A_{X \times Y})$.
\end{proof}

\begin{proposition}\label{prop:prodgen}
Let $\G_X$ and $\G_Y$ be a compact generator of $\Der_\qcoh(\gr \A_X)$ and $\Der_\qcoh(\gr \A_Y)$. Then the object $\iota(\G_X) \cubtimes \iota(\G_Y)$ is a compact generator of $\Der_\qcc(\A_{X \times Y})$.
\end{proposition}

\begin{proof}
The compactness of $\iota(\G_X) \cubtimes \iota(\G_Y)$ follows from Lemma \ref{lem:compactness}. It remains to verify that $\iota(\G_X) \cubtimes \iota(\G_Y)$ is a generator of $\D_\qcc(\A_{X \times Y})$. We have an equivalence
\[
\gr(\iota(\G_X) \cubtimes \iota(\G_Y))\simeq \gr \iota \G_X \boxtimes \gr \iota \G_Y.
\]
It follows from Proposition \ref{prop:presercomgen} that $\gr \iota \G_X$ (resp. $\gr \iota \G_Y$) is a compact generator of $\D_{\qcoh}(\gr \A_X)$ (resp. of $\D_{\qcoh}(\gr \A_Y)$). Thus by \cite[Lemma 3.4.1]{BVdB}, $\gr \iota \G_X \boxtimes \gr \iota \G_Y$ is a compact generator of $\D_\qcoh(\gr(\A_{X \times Y}))$.

Let $\M \in \D_\qcc(\A_{X \times Y})$ such that $\Map_{\A_{X \times Y}}(\iota(\G_X) \cubtimes \iota(\G_Y),\M)=0$.
Applying the functor $\gr$ to the enriched mapping space, we find that
\begin{align*}
\gr \uMap_{\A_{X \times Y}}(\iota(\G_X) \cubtimes \iota(\G_Y),\M) & \simeq \uMap_{\gr \A_{X \times Y}}(\gr (\iota(\G_X) \cubtimes \iota(\G_Y)),\gr \M)\\
                                                               & \simeq \uMap_{\gr \A_{X \times Y}}(\gr \iota(\G_X) \cubtimes \gr \iota(\G_Y)),\gr \M).\\ 
\end{align*}
It follows that $\gr \M \simeq 0$. Since $\M$ is cohomologically complete it follows that $\M \simeq 0$.
\end{proof}
\begin{lemma}\label{lem:eqcom}
The functor
$i^\land \colon \Der_\qcc(\A_X) \otimes_{\Der(\Chbar)} \Der_\qcc(\A_Y) \to \Der_\qcc(\A_X) \otimes_{\Der_\cc(\Chbar)} \Der_\qcc(\A_Y) $ is an equivalence of $\i$-categories.
\end{lemma}

\begin{proof}
This follows immediately from Lemma \ref{prop:compcompgen}.
\end{proof}

\begin{proposition}\label{prop:kunneth}
Let $X_1$ and $X_2$ be two smooth complex algebraic varieties endowed with DQ-algebroids $\A_1$ and $\A_2$. Let $\M_i, \N_i \in \Der^{\omega}_\qcc(\A_i)$ ($i=1, \, 2$). Then,
\begin{equation*}
\uMap_{\A_1}(\M_1,\N_1) \otimes \uMap_{\A_2}(\M_2,\N_2)  \simeq \uMap_{\A_{12}}(\M_1 \ubtimes \M_2,\N_1 \ubtimes \N_2).
\end{equation*}
\end{proposition}
\begin{proof}
\begin{align*}
\uMap_{\A_1}(\M_1,\N_1) \otimes \uMap_{\A_2}(\M_2,\N_2) & \simeq a_{1\ast} \fMap_{\A_1}(\M_1,\N_1) \otimes a_{2\ast} \fMap_{\A_2}(\M_2,\N_2)\\
&\to a_{1 \ast} \left( \fMap_{\A_1}(\M_1,\N_1) \otimes a_{1}^{-1} a_{2\ast} \fMap_{\A_2}(\M_2,\N_2) \right)\\
&\to a_{1 \ast} \left( \fMap_{\A_1}(\M_1,\N_1) \otimes p_{1 \ast} p_{2}^{-1} \fMap_{\A_2}(\M_2,\N_2) \right)\\
&\to a_{1 \ast} p_{1 \ast} \left( p_{2}^{-1} \fMap_{\A_1}(\M_1,\N_1) \otimes  p_{2}^{-1} \fMap_{\A_2}(\M_2,\N_2) \right)\\
&\simeq a_{12 \ast} \left( \left(  \fMap_{\A_1}(\M_1,\A_1) \tensor_{\A_1} \N_1 \right) \btimes \left( \fMap_{\A_2}(\M_2,\A_2) \tensor_{\A_2} \N_2 \right) \right)\\
&\simeq a_{12 \ast} \left( \left(  \fMap_{\A_1}(\M_1,\A_1) \btimes \fMap_{\A_2}(\M_2,\A_2) \right) \! \tensor_{\A_1 \btimes \A_2} \! \left( \N_1 \btimes \N_2 \right) \right)\\ 
&\to a_{12 \ast} \left( \left(  \fMap_{\A_1}(\M_1,\A_1) \ubtimes \fMap_{\A_2}(\M_2,\A_2) \right) \tensor_{\A_{12}} \left( \N_1 \ubtimes \N_2 \right) \right)\\
&\simeq a_{12 \ast} \left( \fMap_{\A_{12}}(\M_1 \ubtimes \M_2,\A_{12}) \tensor_{\A_{12}} \left( \N_1 \ubtimes \N_2 \right) \right)\\
&\simeq a_{12 \ast} \fMap_{\A_{12}}(\M_1 \ubtimes \M_2,\N_1 \ubtimes \N_2)\\
&\simeq \uMap_{\A_{12}}(\M_1 \ubtimes \M_2,\N_1 \ubtimes \N_2).
\end{align*}
The source and target of the total composite are cohomologically complete. Indeed, the source is of uniform $\hbar$-torsion, while the cohomological completeness of the target follows by observing that $\N_1 \ubtimes \N_2$ is cohomologically complete by \cite[Theorem 1.6.1]{KS3} and then applying \cite[Proposition 1.5.10]{KS3} and \cite[Proposition 1.5.12]{KS3} 

Applying the $\gr$ functor to the above morphism and using \cite[Proposition 1.4.4]{KS3} and Proposition \ref{prop:isogr}, we obtain a map
\[
\uMap_{\gr\A_1}\!(\gr \M_1, \gr \N_1) \otimes \uMap_{\gr\A_2}\!(\gr\M_2,\gr\N_2)
\simeq\uMap_{\gr \A_{12}}\!(\gr \M_1 \btimes \gr \M_2,\gr \N_1 \btimes \gr \N_2)
\]
which is an equivalence by \cite[Proposition 4.6]{BFN}. As the functor $\gr$ is conservative on the category of cohomologically complete modules, the morphism is an isomorphism.
\end{proof}

\begin{theorem}\label{thm:monoidality}
The functor $\cubtimes \colon \Der_\qcc(\A_X) \otimes_{\Der_\cc(\Chbar)} \Der_\qcc(\A_Y) \to \Der_\qcc(\A_{X \times Y})$ is an equivalence.
\end{theorem}

\begin{proof} 
Let $\G_X$ and $\G_Y$ be  compact generators of $\Der_\qcoh(\gr \A_X)$ and $\Der_\qcoh(\gr \A_Y)$. Hence, $\iota(\G_X)$ is a compact generator of $\Der_\qcc(\A_X)$ and $\iota(\G_Y)$ is a compact generator of $\Der_\qcc(\A_Y)$.  
It follows from Proposition \ref{lem:gentensproduct} that the object $\iota(\G_X) \tensor \iota(\G_Y)$ is a compact generator of $\D_\qcc(\A_X) \otimes_{\Der(\Chbar)} \D_\qcc(\A_Y)$ and from Lemma \ref{lem:eqcom} that $\iota(\G_X) \ctensor \iota(\G_Y)$ is a compact generator of $\Der_\qcc(\A_X) \otimes_{\Der_\cc(\Chbar)} \Der_\qcc(\A_Y)$.
Moreover, the image of $\iota(\G_X) \ctensor \iota(\G_Y)$ under $\cubtimes$ is $\iota(\G_X) \cubtimes \iota(\G_Y)$, which is a compact generator of $\Der_\qcc(\A_{X \times Y})$ by Proposition \ref{prop:prodgen}.

By Lemma \ref{lem:gentensproduct}, there is a canonical equivalence
\begin{equation*}
\uMap(\iota(\G_X) \ctensor \iota(\G_Y),\iota(\G_X) \ctensor \iota(\G_Y)) \simeq \uMap_{\A_X}(\iota(\G_X),\iota(\G_X)) \ctensor \uMap_{\A_Y}(\iota(\G_Y),\iota(\G_Y)).
\end{equation*}
As $\iota(\G_X)$ and $\iota(\G_Y)$ are of uniform $\hbar$-torsion so are $\uMap_{\A_X\!}(\iota(\G_X),\iota(\G_X))$ and $\uMap_{\A_Y \!}(\iota(\G_Y),\iota(\G_Y))$. Hence, it follows from Lemma \ref{lem:unicomp} that
\begin{equation*}
\uMap(\iota(\G_X) \ctensor \iota(\G_Y),\iota(\G_X) \ctensor \iota(\G_Y)) \simeq \uMap_{\A_X}(\iota(\G_X),\iota(\G_X)) \otimes_{\Chbar} \uMap_{\A_Y}(\iota(\G_Y),\iota(\G_Y)).
\end{equation*}

Hence it suffices to show that the canonical map
\begin{equation*}
\uMap_{\A_{X}}(\iota(\G_X),\iota(\G_X)) \otimes_\Chbar \uMap_{\A_{Y}}(\iota(\G_Y),\iota(\G_Y)) \to \uMap_{\A_{X \times Y}}(\iota(\G_X) \cubtimes \iota(\G_Y),\iota(\G_X) \cubtimes \iota(\G_Y))
\end{equation*}
is an isomorphism in $\D_\cc(\CC^\hbar)$, which is a consequence of Proposition \ref{prop:kunneth}.
\end{proof}

\begin{proposition}
The $\D_\cc(\Chbar)$-module $\D_\qcc(\A_X)$ is dualizable and its dual is $\D_\qcc(\A_{X^\mathrm{a}})$. 
\end{proposition}

\begin{proof}
We exhibit evaluation and coevaluation maps.
Consider the functor
\[
\varepsilon \colon \D_\qcc(\A_{X^\mathrm{a}}) \times \D_\qcc(\A_{X}) \to \D_\cc(\Chbar), \; (\M, \N) \mapsto \Gamma(X;\widehat{\M \tensor_{\A_X} \N}). 
\]
By Proposition \ref{prop:commlim}, this functor commutes with colimits, inducing the evaluation
\[
\varepsilon \colon \D_\qcc(\A_{X^\mathrm{a}}) \tensor_{\D_\cc(\Chbar)} \D_\qcc(\A_{X}) \to \D_\cc(\Chbar)
\]%
which we again denote $\varepsilon$.
We also have the evaluation map
\[
 \mathrm{ev} : \D_\qcc(\A_{X^\mathrm{a} \times X}) \to \D_\cc(\Chbar), \; \M \mapsto \A_{\Delta} \convp_{X^\mathrm{a} \times X} \M.
\]
This leads to the following commutative diagram
\[
\xymatrix{ \D_\qcc(\A_{X^\mathrm{a}}) \tensor_{\D_\cc(\Chbar)} \D_\qcc(\A_{X})  \ar[r]^-{\varepsilon} \ar[d]_-{\wr} & \D_\cc(\Chbar)\\
\D_\qcc(\A_{X^\mathrm{a} \times X}) \ar[ru]_-{\mathrm{ev}}. &\\
}
\]

For the coevaluation, since $\D_\qcc(\A_{X \times X^\mathrm{a}})$ is tensored over $\D_\cc(\Chbar)$, we have a functor
\begin{equation}\label{map:precoeval}
\mathrm{coev} \colon \D_\cc(\Chbar) \to \D_\qcc(\A_{X \times X^\mathrm{a}}), \; M \mapsto M \ctensor \A_\Delta.
\end{equation}
By Theorem \ref{thm:monoidality}, we have the identification
\[
\D_\qcc(\A_X) \tensor_{\D_\cc(\Chbar)}  \D_\qcc(\A_{X^\mathrm{a}}) \simeq\D_\qcc(\A_{X \times X^\mathrm{a}}).
\]
Composing the map \eqref{map:precoeval} with the above identification we obtain the coevalution
\[
\eta: \D_\cc(\Chbar) \to \D_\qcc(\A_X) \tensor_{\D_\cc(\Chbar)} \D_\qcc(\A_{X^\mathrm{a}}).
\]
We now check that $\varepsilon$ and $\eta$ satisfy the triangle identities. All the tensor products will be over $\D_\cc(\Chbar)$ but we will omit it for brevity. We check that
\begin{equation*}
\D_\cc(\CC^\hbar) \otimes \D_\qcc(\A_X) \stackrel{\eta \otimes \id}{\to} \D_\qcc(\A_X) \otimes \D_\qcc(\A_{X^\mathrm{a}}) \otimes \D_\qcc(\A_X) \stackrel{\id \otimes \varepsilon}{\to}  \D_\qcc(\A_X) \otimes \D_\cc(\CC^\hbar)
\end{equation*}
is the identity.

For readability, we set $X_1=X_2=X_3$. We denote by $\Delta_{12}$ the diagonal of $X_1 \times X_2$ and  by $\Delta_{23}$ the diagonal of $X_2 \times X_3$. We define the following functors
\[
g:\D_\qcc(\A_{X_1 \times X_{2^\mathrm{a}} \times X_3}) \to \D_\qcc(\A_{X_1}), \; \M \mapsto \M \convp_{X_2 \times X_{3^\mathrm{a}}} \A_{\Delta_{23}}.
\]
\begin{equation*}
f:\D_\qcc(\A_{X_3}) \to \D_\qcc(\A_{X_1 \times X_{2^\mathrm{a}} \times X_3}), \; \A_{\Delta_{12}} \cubtimes \M.
\end{equation*}
For the sake of compactness we write $\C_X$ for $\D_\qcc(\A_X)$, $\C_{X^\mathrm{a}}$ for $\D_\qcc(\A_{X^\mathrm{a}})$,
$\C_{X \times X^\mathrm{a}}$ for $\D_\qcc(\A_{X \times X^\mathrm{a}})$ and $\dsu_X$ for $\D_\cc(\Chbar_X)$.
Then, we have the following commutative diagram
\[
\xymatrix{
\dsu \times \C_{X_3} \ar[dd] \ar[dr]_-{\mathrm{coev} \times \id} \ar@{-->}[r] & \C_{X_1} \times \C_{X_2^\mathrm{a}} \times \C_{X_3} \ar[d]^-{\cubtimes \times \id} \ar[rr]^-{\id \times \varepsilon} \ar[rd]^-{\id \times \cubtimes} & & \C_{X_1} \times \dsu \ar[dd]\\
                         & \C_{X_1 \times X_2^\mathrm{a}} \times \C_{X_3} \ar[d]_-{\id \times \cubtimes} & \C_{X_1} \times \C_{X_2^\mathrm{a}\times X_3} \ar[ru]^-{\id \times \mathrm{ev}} \ar[ld]_-{\cubtimes \times \id} &\\ 
\C_{X_3} \ar[r]_-f & \C_{X_1 \times X_2^\mathrm{a} \times X_3} \ar[rr]_-g & & \C_{X_1}
}
\]
and we have
\[
g \circ f(\M)=(\A_{\Delta_{12}} \cubtimes \M) \convp_{X_2 \times X_3^\mathrm{a}} \A_{\Delta_{23}}\simeq (\A_{\Delta_{12}} \convp_{X_2} \A_{\Delta_{23}} \convp_{X_3}\M)\simeq \A_{\Delta_{13}} \convp_{X_3}\M\simeq \M.
\]
It follows that $g \circ f$ is equivalent to the identity.
The second triangular identity is similar.
\end{proof}

\begin{proposition}
The functor $\varepsilon$ induces a functor $\D_\qcc(\A_{X^\mathrm{a}}) \to \Fun_{\D_\cc(\CC^\hbar)}(\D_\qcc(\A_X),\D_\cc(\Chbar))$ given by $\M \mapsto ( \F \mapsto \Gamma(X,\M \ctensor_{\A_X} \F))$, which is an equivalence of $\Der_\cc(\CC^\hbar)$-modules.
\end{proposition}

\begin{proof}
This follows from the definition of a dual.
\end{proof}

\begin{proposition}
The canonical functor
\[
\Fun_{\D_\cc(\CC^\hbar)}(\D_\qcc(\A_X),\D_\cc(\CC^\hbar)) \otimes_{\D_\cc(\CC^\hbar)}\D_\qcc(\A_Y) \to \Fun_{\D_\cc(\CC^\hbar)}(\D_\qcc(\A_X),\D_\qcc(\A_Y))\\
\]
is an equivalence.
\end{proposition}

\begin{proof}
This follows from the fact that $\D_\qcc(\A_X)$ is a dualizable $\D_\cc(\CC^\hbar)$-module.
\end{proof}

\subsection{Integral representations}

We now establish some integral representation theorems for DQ-modules.

\begin{theorem}\label{thm:FMDQ}
Let $X$ and $Y$ be two smooth algebraic varieties endowed with DQ-algebroid stacks $\A_X$ and $\A_Y$. There is an equivalence of $\i$-categories
\begin{equation}\label{fun:inttrans}
\Phi_{(-)} \colon \D_\qcc(\A_{Y \times X^\mathrm{a}}) \stackrel{\sim}{\longrightarrow} \Fun_{\D_\cc(\CC^\hbar)}(\D_\qcc(\A_X),\D_\qcc(\A_Y)).
\end{equation}
\end{theorem}\label{thm:intrep}

\begin{proof}
We have the following sequence of equivalences in $\Mod_{\D_\cc(\CC^\hbar)}$:
\begin{align*}
\Fun_{\D_\cc(\CC^\hbar)}(\D_\qcc(\A_X),\D_\qcc(\A_Y)) &\simeq
\D_\qcc(\A_Y) \tensor_{\D_\cc(\CC^\hbar)}
\Fun_{\D_\cc(\CC^\hbar)}(\D_\qcc(\A_X),\D_\cc(\CC^\hbar))\\
&\simeq \D_\qcc(\A_{Y})\tensor_{\D_\cc(\CC^\hbar)}\D_\qcc(\A_{X^\mathrm{a}})\\ 
&\simeq \D_\qcc(\A_{Y \times X^\mathrm{a}}).
\end{align*}
\end{proof}
The functor $\Phi_{(-)}$ is the composite of the top horizontal rows of diagrams \eqref{diag:phi1} and \eqref{diag:phi2}:
\begin{equation}\label{diag:phi1}
\scalebox{0.85}{
\xymatrix{ \D_\qcc(\A_{Y \times X^\mathrm{a}}) \ar[r]^-{{\cubtimes}^{-1}} 
&  \D_\qcc(\A_Y) \tensor_{\D_\cc(\Chbar)} \D_\qcc(\A_{X^\mathrm{a}}) \ar[r]^-{\id \otimes \gamma} 
& \D_\qcc(\A_Y) \tensor_{\D_\cc(\CC^\hbar)} \Fun_{\D_\cc(\CC^\hbar)}(\D_\qcc(\A_X),\D_\cc(\CC^\hbar))\\
\D_\qcc(\A_{Y \times X^\mathrm{a}}) \ar@{=}[u] 
&  \D_\qcc(\A_Y) \times \D_\qcc(\A_{X^\mathrm{a}}) \ar[r]^-{\id \times \gamma} \ar[l]_-{\cubtimes} \ar[u] 
&  \D_\qcc(\A_Y) \times \Fun_{\D_\cc(\CC^\hbar)}(\D_\qcc(\A_X),\D_\cc(\CC^\hbar)) \ar[u]
}}
\end{equation}
\begin{equation}\label{diag:phi2}
\scalebox{0.85}{
\xymatrix{
\D_\qcc(\A_Y) \tensor_{\D_\cc(\CC^\hbar)} \Fun_{\D_\cc(\CC^\hbar)}(\D_\qcc(\A_X),\D_\cc(\CC^\hbar)) \ar[r]
& \Fun_{\D_\cc(\CC^\hbar)}(\D_\qcc(\A_X),\D_\qcc(\A_Y))
\\
  \D_\qcc(\A_Y) \times \Fun_{\D_\cc(\CC^\hbar)}(\D_\qcc(\A_X),\D_\cc(\CC^\hbar)) \ar[r]^-{\nu} \ar[u]
& \Fun_{\D_\cc(\CC^\hbar)}(\D_\qcc(\A_X),\D_\qcc(\A_Y)) \ar@{=}[u]
}}
\end{equation}
\begin{lemma}\label{lem:basechange} Let $\M \in \D_\qcc(\A_{X^\mathrm{a}}), \F \in \D_\qcc(\A_{X}), \N \in \D_\qcc(\A_{Y})$.  There is a natural isomorphism in $\D_\qcc(\A_Y)$
\begin{equation}\label{map:basechange}
a_Y^{-1} a_{X \ast} (\M \ctensor_{\A_X} \F) \ctensor \N \stackrel{\sim}{\longrightarrow} p_{Y \ast} (p_{X}^{-1}(\M \ctensor_{\A_X} \F) \ctensor p^{-1}_Y \N).
\end{equation}
\end{lemma}

\begin{proof}
We have the following pull-back diagram
\begin{equation*}
\xymatrix{ X \times Y \ar[r]^-{p_Y} \ar[d]_-{p_X} & Y \ar[d]^-{a_Y}\\
           X \ar[r]^-{a_X} & \mathrm{pt}
}
\end{equation*}
and the following morphism
\begin{align*}
&\Map_{\CC^\hbar_X}(a_X^{-1} a_{X \ast} (\M \ctensor_{\A_X} \F),\M \ctensor_{\A_X} \F)\\   
\rightarrow &\Map_{\D_\qcc(p_Y^{-1}\A_Y)}(p_X^{-1}a_X^{-1} a_{X \ast} (\M \ctensor_{\A_X} \F) \ctensor p_{Y}^{-1}\N, (p_{X}^{-1}(\M \ctensor_{\A_X} \F) \ctensor p^{-1}_Y \N))\\
\simeq &\Map_{\D_\qcc(\A_Y)}(a_Y^{-1} a_{X \ast} (\M \ctensor_{\A_X} \F) \ctensor \N, p_{Y \ast} (p_{X}^{-1}(\M \ctensor_{\A_X} \F) \ctensor p^{-1}_Y \N)).
\end{align*}
The image of the co-unit of the adjunction $(a_X^{-1},a_{X \ast})$ provides the map \eqref{map:basechange}. Applying the $\gr$ functor to the morphism \eqref{map:basechange}, we get the map

\begin{equation}\label{map:basegriso}
a_Y^{-1} a_{X \ast} (\gr \M \ctensor_{\O_X} \gr \F) \tensor \gr \N \stackrel{\sim}{\longrightarrow} p_{Y \ast} (p_{X}^{-1}(\gr \M \tensor_{\O_X} \gr \F) \tensor p^{-1}_Y \gr \N).
\end{equation}
This map fits into the following commutative diagram
\begin{equation*}
\xymatrix{
a_Y^{-1} a_{X \ast} (\gr \M \ctensor_{\O_X} \gr \F) \tensor \gr \N \ar[r] \ar[d]^-{\wr} & p_{Y \ast} (p_{X}^{-1}(\gr \M \tensor_{\O_X} \gr \F) \tensor p^{-1}_Y \gr \N) \ar[d]^-{\wr}\\
a_Y^{\ast} a_{X \ast} (\gr \M \ctensor_{\O_X} \gr \F) \tensor_{\O_Y} \gr \N \ar[r]^-{\sim}  & p_{Y \ast} (p_{X}^{\ast}(\gr \M \tensor_{\O_X} \gr \F) \tensor_{\O_{X \times Y}} p^{\ast}_Y \gr \N).
}
\end{equation*}
where the bottom arrow is an isomorphism by the flat base change formula (see \cite[\href{https://stacks.math.columbia.edu/tag/08ET}{Section 08ET}]{stacks-project}). This implies that the map \eqref{map:basegriso} is an isomorphism. Since $\gr$ is conservative, we deduce that the morphism \eqref{map:basechange} is an isomorphism.
\end{proof}
Diagrams \eqref{diag:phi1}, \eqref{diag:phi2} and Lemma \ref{lem:basechange} imply that 
\begin{align*}
\Phi_{(\M \cubtimes \N)} & \simeq \nu \circ (\gamma \times \id) (\M,\N)\\
                                     &\simeq \nu(\Gamma(X,\M \ctensor_{\A_X} (-)),\N)\\
                                     &\simeq a_Y^{-1}(\Gamma(X,\M \ctensor_{\A_X} (-))) \ctensor \N\\
                                     &\simeq p_{Y \ast} (p_X^{-1}(\M \ctensor_{\A_X} (-)) \ctensor p_Y^{-1}\N)\\
                                     &\simeq p_{Y \ast} ( (\M \cubtimes \N) \ctensor_{p_X^{-1}\A_X} p_X^{-1}(-)).             
\end{align*}
It follows from Theorem \ref{thm:monoidality} that the category $\D_{\qcc}(\A_{X^\mathrm{a} \times Y})$ is generated under colimits by objects of the form $\M \cubtimes \N$ and $\Phi_{(-)}$ is a functor in $\PrL$. Thus, it commutes with colimits. As
\[
 \Phi_{(\M \cubtimes \N)} \simeq p_{Y \ast} ( (\M \cubtimes \N) \ctensor_{p_X^{-1}\A_X}p_X^{-1}(-)),
\]
it follows that for, every $\K \in \D_{\qcc}(\A_{X^\mathrm{a} \times Y})$,
\begin{equation}\label{eq:FMformula}
\Phi_{\K} \simeq p_{Y \ast} ( \K \ctensor_{p_X^{-1}\A_X}p_X^{-1}(-)).
\end{equation}

\begin{corollary}\label{cor:semiequi}
Let $X$ and $Y$ be two smooth algebraic varieties endowed with $DQ$-algebroid stacks $\A_X$  and $\A_Y$. If $\D_{\qcc}(\A_X) \simeq \D_{\qcc}(\A_Y)$ in $\Mod_{\Der(\CC^\hbar)}$ then $\D_{\qcoh}(\O_X) \simeq \D_{\qcoh}(\O_Y)$ in $\Mod_{\Der(\CC)}$.
\end{corollary}

\begin{proof}
Let $F \colon \D_\qcc(\A_X) \to \D_\qcc(\A_X)$ be an equivalence of $\Der(\CC^\hbar)$-modules. It follows from Theorem \ref{thm:FMDQ} and formula \eqref{eq:FMformula} that $F$ is an integral transform with kernel $\K$. It follows from \cite[Theorem 3.16]{FMQ} that the integral transform $\Phi_{\gr \K} \colon \D_{\qcoh}(\O_X) \to \D_{\qcoh}(\O_Y)$ is an equivalence of categories.
\end{proof}

\begin{theorem}\label{thm:Intrepcomp}
	Let $X$ and $Y$ be two smooth and proper algebraic varieties endowed with DQ-algebroid stacks $\A_X$ and $\A_Y$. The functor
	\begin{equation*}
	\Phi_{(-)} \colon \D_\qcc(\A_{Y \times X^\mathrm{a}}) \stackrel{\sim}{\longrightarrow} \Funl_{\Der(\CC^\hbar)}(\D_\qcc(\A_X),\D_\qcc(\A_Y))
	\end{equation*}
    in \eqref{fun:inttrans} induces an equivalence
    
	\begin{equation*}
	\widetilde{\Phi}_{(-)} \colon \Derb_\coh(\A_{Y \times X^\mathrm{a}})\stackrel{\sim}{\longrightarrow} \Funl_{\Der(\CC^\hbar), \omega}(\Der_\qcc(\A_X),\Der_\qcc(\A_Y)).
	\end{equation*}
\end{theorem}

\begin{proof}
	We consider the restriction of the equivalence 
	\[
	\Phi{(-)} \colon \D_\qcc(\A_{Y \times X^\mathrm{a}}) \to \Fun_{\Der(\CC^\hbar)}(\D_\qcc(\A_X),\D_\qcc(\A_Y))
	\]
	to the full subcategory $\Derb_\coh(\A_{Y \times X^\mathrm{a}})$ of  $\D_\qcc(\A_{Y \times X^\mathrm{a}})$. It follows from Proposition \ref{prop:cohprecomp} that the restriction of $\Phi_{(-)}$ to $\Derb_\coh(\A_{Y \times X^\mathrm{a}})$ factors through $\Fun_{\Der(\CC^\hbar), \omega}(\D_\qcc(\A_X),\D_\qcc(\A_Y))$. Hence, we obtain a functor
	\begin{equation*}
	\widetilde{\Phi}(-) \colon \Derb_\coh(\A_{Y \times X^\mathrm{a}}) \to \Fun_{\D_\cc(\CC^\hbar), \omega}(\D_\qcc(\A_X),\D_\qcc(\A_Y)).
	\end{equation*}
The functor $\Phi$ is a morphism in $\Mod_{\Der(\CC^\hbar)}(\PrL)$ as it is obtained as a composition of morphisms of $\Mod_{\Der(\CC^\hbar)}(\PrL)$. Moreover, the $\i$-categories $\Derb_\coh(\A_{Y \times X^\mathrm{a}})$ and $\Fun_{\Der(\CC^\hbar), \omega}(\D_\qcc(\A_X),\D_\qcc(\A_Y))$ are objects of $\Mod_{\D^\omega(\CC^\hbar)}( \Cat_{\i,\mathrm{idem}}^\mathrm{ex})$ respectively by Lemma \ref{lem:linstruccoh} and Corollary \ref{lem:struclin}. Hence,  $\widetilde{\Phi}$ is a morphism in $\Mod_{\D^\omega(\CC^\hbar)}( \Cat_{\i,\mathrm{idem}}^\mathrm{ex})$. 

	Since $\Derb_\coh(\A_{Y \times X^\mathrm{a}})$ is a full subcategory of $\D_\qcc(\A_{Y \times X^\mathrm{a}})$, $\widetilde{\Phi}$ is fully faithful. Let us show that it is essentially surjective. Let $F \in \Fun_{\Der(\CC^\hbar), \omega}(\D_\qcc(\A_X),\D_\qcc(\A_Y))$. It follows from Theorem \ref{thm:FMDQ} that there exists $\K \in \Der_\qcc(\A_{X \times Y^\mathrm{a}})$ such that $F \simeq \Phi_\K$. As $F$ preserves compact objects, $\Phi_\K$ does as well.
	It follows from Theorem \ref{thm:compactkernel} that $\K \in \Derb_\coh(\A_{Y \times X^\mathrm{a}})$, completing the proof.
\end{proof}

\begin{corollary}\label{cor:intcoh} Let $X$ and $Y$ be two smooth and proper algebraic varieties endowed with DQ-algebroid stacks $\A_X$ and $\A_Y$. Then
\begin{equation*}
\Derb_\coh(\A_{Y \times X^\mathrm{a}}) \simeq\Fun_{\Der^\omega(\CC^\hbar), }(\D^\omega_\qcc(\A_X),\D^\omega_\qcc(\A_Y))
\end{equation*}	
\end{corollary}

\begin{proof}
It follows from Theorem \ref{thm:Intrepcomp} that we have the equivalence
	\begin{equation*}
	\widetilde{\Phi}(-) \colon \Derb_\coh(\A_{Y \times X^\mathrm{a}})\stackrel{\sim}{\longrightarrow} \Fun_{\Der(\CC^\hbar), \omega}(\D_\qcc(\A_X),\D_\qcc(\A_Y)).
	\end{equation*}
	Moreover, Lemma \ref{lem:indsym} implies that we also have the equivalence
	\begin{equation*}
	\Fun_{\Der(\CC^\hbar), \omega}(\D_\qcc(\A_X),\D_\qcc(\A_Y)) \simeq \Fun_{\Der^\omega(\CC^\hbar)}(\D^\omega_\qcc(\A_X),\D^\omega_\qcc(\A_Y)).
	\end{equation*}
	Hence,
	\begin{equation*}
	\Derb_\coh(\A_{Y \times X^\mathrm{a}}) \simeq\Fun_{\Der^\omega(\CC^\hbar), }(\Der^\omega_\qcc(\A_X),\Der^\omega_\qcc(\A_Y)).
	\end{equation*}
\end{proof}

\begin{theorem}
Let $X$ and $Y$ be two smooth and proper algebraic varieties endowed with DQ-algebroid stacks $\A_X$ and $\A_Y$.
There is an equivalence
\begin{equation*}
	\Derb_\coh(\A_{Y \times X^\mathrm{a}}) \simeq \Fun^\mathrm{L}_{\Der^\omega(\CC^\hbar)}(\Derb_\coh(\A_X),\Derb_\coh(\A_Y)).
\end{equation*}
\end{theorem}

\begin{proof}
Recall that $\Der_\qcc(\A_X)\simeq \Ind(\Der^\omega_\qcc(\A_X))$.
We denote by $j \colon \Derb_\coh(\A_X) \subset \Der_\qcc(\A_X)$ the fully faithful inclusion of the $\i$-category of coherent DQ-modules into the $\i$-category of qcc DQ-modules.
Passing to $\Ind$-objects induces an equivalence of $\i$-categories
\begin{equation*}
\Fun_{\Der^\omega(\Chbar)}(\Der^\omega_\qcc(\A_X),\Der^\omega_\qcc(\A_Y)) \stackrel{\sim}{\longrightarrow} \Fun^\mathrm{L}_{\Der(\Chbar),\omega}(\Der_\qcc(\A_X),\Der_\qcc(\A_Y)),
\end{equation*}
and $j$ induces a functor

\begin{equation*}
 \Fun^\mathrm{L}_{\Der(\Chbar),\omega}(\Der_\qcc(\A_X),\Der_\qcc(\A_Y)) \stackrel{j^\ast}{\longrightarrow} \Fun_{\Der^\omega(\Chbar)}(\Derb_\coh(\A_X),\Der_\qcc(\A_Y)).
\end{equation*}
We know by Theorem \ref{thm:Intrepcomp}, that the objects of the $\i$-category $\Fun^\mathrm{L}_{\Der(\Chbar),\omega}(\Der_\qcc(\A_X),\Der_\qcc(\A_Y))$ correspond to Fourier-Mukai transforms with coherent kernels. Hence, the restriction of such a functor to $\Derb_\coh(\A_X)$ induces a $\Der^\omega(\Chbar)$-linear functor from $\Derb_\coh(\A_X)$ to $\Derb_\coh(\A_Y)$. Furthermore, such a Fourier-Mukai transform $\Phi_\K \colon \Derb_\coh(\A_X) \to \Derb_\coh(\A_Y)$ has a right adjoint, as shown in the proof of \cite[Proposition 3.15]{FMQ}. Thus $j_\ast$ factors through the full subcategory
\[
\Fun^\mathrm{L}_{\Der^\omega(\CC^\hbar)}(\Derb_\coh(\A_X),\Derb_\coh(\A_Y))\subset\Fun_{\Der^\omega(\CC^\hbar)}(\Derb_\coh(\A_X),\Derb_\coh(\A_Y))
\]
and we get a functor
\begin{align*}
\alpha \colon \Fun_{\Der^\omega(\Chbar)}(\Der^\omega_\qcc(\A_X),\Der^\omega_\qcc(\A_Y)) & \stackrel{\Ind}{\longrightarrow} \Fun^\mathrm{L}_{\Der(\Chbar),\omega}(\Der_\qcc(\A_X),\Der_\qcc(\A_Y)) \\ 
&\stackrel{j^\ast}{\longrightarrow} \Fun^\mathrm{L}_{\Der^\omega(\CC^\hbar)}(\Derb_\coh(\A_X),\Derb_\coh(\A_Y)).
\end{align*}
Hence, $\alpha:=j^\ast \circ \Ind$. The fully faithful inclusion
$i \colon \Der^\omega_\qcc(\A_X) \subset \Derb_\coh(\A_X)$ induces a functor
\begin{equation}\label{fun:prebeta}
\Fun^\mathrm{L}_{\Der^\omega(\CC^\hbar)}(\Derb_\coh(\A_X),\Derb_\coh(\A_Y)) \stackrel{i^\ast}{\to}\Fun_{\Der^\omega(\Chbar)}(\Der^\omega_\qcc(\A_X),\Derb_\coh(\A_Y))
\end{equation}
If $F \in \Fun^\mathrm{L}_{\Der^\omega(\CC^\hbar)}(\Derb_\coh(\A_X),\Derb_\coh(\A_Y))$ and $\M \in \Derb_\coh(\A_X)$ is of $\hbar$-torsion, then the $\Der^\omega(\CC^\hbar)$-linearity of $F$ implies that $F(\M)$ is again of $\hbar$-torsion. It follows from Theorem \ref{thm:compactobj} that the functor \eqref{fun:prebeta} factors through $\Fun_{\Der^\omega(\Chbar)}(\Der^\omega_\qcc(\A_X),\Der^\omega_\qcc(\A_Y))$ and we get a functor
\begin{equation}
\beta \colon \Fun^\mathrm{L}_{\Der^\omega(\CC^\hbar)}(\Derb_\coh(\A_X),\Derb_\coh(\A_Y)) \stackrel{i^\ast}{\to}\Fun_{\Der^\omega(\Chbar)}(\Der^\omega_\qcc(\A_X),\Der^\omega_\qcc(\A_Y)).
\end{equation}

It is straightforward to check that $\beta\circ\alpha(F) = \Ind(F) \circ j \circ i \simeq F$ and $\alpha \circ \beta (G) \circ i \simeq \Ind(G \circ i)\circ j \circ i\simeq G \circ i$.
Moreover, any object $\M$ of $\Derb_\coh(\A_X)$ can be written as a filtered colimit of objects of $\Der^\omega_\qcc(\A_X)$ and these colimits are also colimits in $\Der_\qcc(\A_X)$. The functor $G$ commutes with colimits by hypothesis and $\alpha \circ \beta (G)$ commutes with colimits of the above type. Hence
$\alpha \circ \beta (G) \simeq G$.
\end{proof}
\appendix

\section{Local and complete \texorpdfstring{$\i$}{infinity}-categories}\label{app:nilloccomp}

\subsection{General results}

In this section, we briefly review some of the material of \cite[Chapter 7]{SAG} and establish a few related results.
The results \ref{lem:indsym} and \ref{lem:struclin} are extracted from \cite{HTT}, \cite{BFN} and \cite{BGH}.
\begin{lemma}\label{lem:indsym}
The functor $(-)^\omega:\PrL_{,\omega}\to\Cat_{\i,\mathrm{idem}}^\mathrm{ex}$ is an equivalence of symmetric monoidal $\i$-categories. In particular, if $R$ is a commutative ring (or commutative ring spectrum), then  $(-)^\omega:\Mod_{\Der(R)}(\PrL_{,\omega})\to \Mod_{\Der(R)}(\Cat_{\i,\mathrm{idem}}^\mathrm{ex})$ is an equivalence of $\i$-categories. 
\end{lemma}

\begin{lemma} \label{lem:struclin}
Let $\C$ and $\D$ be two objects of $\Mod_{\Der(R)}(\PrL_{,\omega})$ where $R$ is an $\EE_n$-algebra. We write $\Fun_{\Der(R),\,\omega}(\C,\D)$ for the full subcategory of $\Fun_{\Der(R)}(\C,\D)$ spanned by the functors preserving compact objects. Then $\Fun_{\Der(R),\,\omega}(\C,\D) \in \Mod_{\Der(R)}(\Cat_{\i,\mathrm{idem}}^\mathrm{ex})$.
\end{lemma}

Throughout this section, $A$ denotes a commutative ring with unit
(or more generally, a connective commutative differential graded algebra, or even an $\EE_\i$-ring) and $I\subset\pi_0 A$ a finitely generated ideal. Let $x \in \pi_0 A$. We denote by
\[
A[x^{-1}]\simeq\colim\{A\overset{x}{\too} A\overset{x}{\too} A\overset{x}{\too}\cdots\}
\]
the $\EE_\i$-ring obtained by inverting $x$.
It comes equipped with an $\EE_\i$-ring map $A\to A[x^{-1}]$ and corepresents the subfunctor of $\Map_{\EE_\i}(A,-)$ consisting of the $\EE_\i$-ring maps which send $x$ to a unit. We refer the reader to \cite[Proposition 7.2.3.27]{HA} for details on $\EE_\i$-rings and their localizations.

We denote by $\Der(A)$ the presentably symmetric monoidal $\i$-category of $A$-module spectra (usually denoted $\Mod_A$ in the algebraic topology literature) and consider the $\i$-category $\Mod_{\Der(A)}(\PrL)$ of (left) $A$-module objects in $\PrL$.

\begin{definition}
An $A$-linear $\i$-category is an object of $\Mod_{\Der(A)}(\PrL)$.
An $A$-linear functor between $A$-linear $\i$-categories is a morphism of (left) $A$-module objects of $\PrL$.
\end{definition}

We typically write $\Mod_{\Der(A)}$ in place of $\Mod_{\Der(A)}(\PrL)$.
Any $A$-linear $\i$-category $\C$ is naturally enriched in $\Der(A)$.
The enrichment can be described informally as follows: by virtue of the action of $\Der(A)$ on $\C$, there is a left adjoint functor $\Der(A)\otimes\C\to\C$.
Using the equivalence
\[
\Der(A)\otimes\C\simeq\Funr(\C^{\op},\Der(A)),
\]
the right adjoint may be regarded as a functor
\[
\C\to\Funr(\C^{\op},\Der(A))\subset\Fun(\C^{\op},\Der(A)).
\]
The enriched mapping space functor is the induced functor
\[
\uMap_\C:\C^{\op}\times\C\to\Der(A).
\]
We refer the reader to \cite[Appendix D.7.1]{SAG} for more details.

\begin{definition}
Let $\C$ be a stable $A$-linear $\i$-category.
\begin{enumerate}
	\item An object $N \in \C$ is said to be $I$-nilpotent if for each $x \in I$, $A[x^{-1}] \otimes_A N\simeq 0$.
	\item An object $L \in \C$ is $I$-local if, for every $I$-nilpotent object $N \in \C$, $\Map_\C(N,L)\simeq 0$.
	\item An object $M \in \C$ is $I$-complete if, for every $I$-local object $L \in \C$, $\Map_\C(L,M)\simeq 0$
\end{enumerate}
We write $\C^{\mathrm{nil}(I)}$, $\C^{\loc(I)}$, and $\C^{\Cpl(I)}$, respectively, for the full subcategory of $\C$ spanned by $I$-nilpotent, $I$-local, and $I$-complete objects of $\C$.
\end{definition}

\begin{notation}
We typically write $\Der_{\Cpl(I)}(A)$ in place of $\Der^{\Cpl(I)}(A)$ for the full subcategory of $\Der(A)$ spanned by the $I$-complete objects of $\Der(A)$. If there no risk of confusion regarding $I$ we write $\Der_{\Cpl}(A)$ instead of $\Der_{\Cpl(I)}(A)$. When $A$ is a $\ZZ[\hbar]$-algebra without $\hbar$-torsion and $I=(\hbar)$, we write $\cc$ instead of $\Cpl(\hbar)$.
\end{notation}

\begin{proposition}\label{prop:fundccnil}
Let $\C$ be a stable $A$-linear $\i$-category and $I\subset\pi_0 A$ a finitely generated ideal.
\begin{enumerate}
\item $\C^{\mathrm{nil}(I)}$ is a full stable $A$-linear subcategory of $\C$. The inclusion $i_\lor\colon\C^{\mathrm{nil}(I)} \subset \C$ is an exact colimit preserving functor with right adjoint $i^{\lor} \colon \C \to \C^{\mathrm{nil}(I)}$. If $\C$ is compactly generated, then $\C^{\nil(I)}$ is also compactly generated and the inclusion $i_{\lor}:\C^{\mathrm{nil}(I)} \subset \C$ preserves compact objects.
\item \label{item:funccnil} $\C^{\loc(I)}$ is a full stable $A$-linear subcategory of $\C$. The inclusion functor $j_*:\C^{\loc(I)} \subset \C$ admits a left adjoint $j^*\colon \C \to \C^{\loc(I)}$ which fits into a functorial exact triangle $i_\lor i^\lor \to \id_\C \to j_*j^*$ so that $(\C^{\mathrm{nil}(I)},\C^{\loc(I)})$ is a semi-orthogonal decomposition of $\C$.
Moreover, the inclusion functor $j_*:\C^{\loc(I)} \subset \C$ also admits a right adjoint which we will denote $j^\times:\C\to\C^{\loc(I)}$.
\item $\C^{\Cpl(I)}$ is a full stable subcategory of $\C$.
The inclusion functor $i_\land:\C^{\Cpl(I)} \subset \C$ admits a left adjoint $i^\land \colon \C \to \C^{\Cpl(I)}$ which fits into a functorial exact triangle $j_*j^\times \to \id_\C \to i_\land i^\land$.
The pair $(\C^{\loc(I)},\C^{\Cpl(I)})$ is a semi-orthogonal decomposition of $\C$.
\end{enumerate}
\end{proposition}

\begin{proof}
The first statement is \cite[Proposition 7.1.1.12]{SAG}.
The second statement is \cite[Proposition 7.2.4.9]{SAG} and \cite[Proposition 7.2.4.4]{SAG}.
The third statement is  \cite[Proposition 7.3.1.4]{SAG} and  \cite[Proposition 7.3.1.5]{SAG}.
\end{proof}

\begin{proposition}\label{prop:nilloc}
Let $\C$ be a (stable) $A$-linear $\i$-category, let $N$ be an $I$-nilpotent object of $\D(A)$, and let $L$ be an $I$-local object of $\D(A)$.
Then for any object $M$ of $\C$, the tensor $N\otimes_A M$ is an $I$-nilpotent object of $\C$, and the tensor $L\otimes_A M$ is an $I$-local object of $\C$.
\end{proposition}

\begin{proof}
Given an element $x\in I$, the localization $A[x^{-1}]\otimes_A \otimes N\otimes_A M\simeq 0$ since $A[x^{-1}]\otimes_A N\simeq 0$.
This proves the first claim.
For the second claim, observe that the functor $L\mapsto L\otimes_A M$ is an $A$-linear functor $\D(A)\to\C$.
Hence by \cite[Proposition 7.2.4.9]{SAG} it follows that $L\otimes_A M$ is $I$-local whenever $L$ is $I$-local.
\end{proof}

\begin{proposition}\label{prop:complcrit}
An object $M \in \C$ is $I$-complete if and only if for each $x \in I$, $M^{A[x^{-1}]}\simeq 0$.
\end{proposition}

\begin{proof}
It follows from \cite[Corollary 7.3.3.3]{SAG} that $M \in \C$ is $I$-complete if and only if it is $(x)$-complete for every $x \in I$. Hence, it sufficient to prove that an object $M$ is $(x)$-complete if and only if $M^{A[{x^{-1}]}}$ vanishes.
Notice that $M^{A[x^{-1}]} \in \C^{\loc(x)}$ as for every $N \in \C^{\nil(x)}$,
	\begin{equation*}
	\Map_\C(N,M^{A[x^{-1}]}) \simeq \Map_\C(A[x^{-1}] \otimes_A N,M)
	\end{equation*}
	is contractible since $A[x^{-1}] \otimes_A N \simeq 0$.

Let $L$ be an $(x)$-local object. We have the following sequence of equivalences
\begin{align*}
\Map_\C(L,M) &\simeq\Map_\C(A[x^{-1}] \otimes_A L,M) \simeq \Map_\C(L,M^{A[x^{-1}]}).
\end{align*}
If $M$ is $(x)$-complete, setting $L=M^{A[x^{-1}]}$ in the above formula implies that the mapping space $\Map_\C(M^{A[x^{-1}]},M^{A[x^{-1}]})$ is contractible.
Consequently the identity of $M^{A[x^{-1}]}$ is homotopic to the zero map, so that $M^{A[x^{-1}]} \simeq 0$.
Conversely, if $M^{A[{x^{-1}]}}$ vanishes, then
\[
\Map_\C(L,M)\simeq\Map_\C(A[x^{-1}]\otimes_A L,M)\simeq\Map_\C(L,M^{A[x^{-1}]})\simeq 0
\]
for every $(x)$-local object $L$.
It follows that $M$ is $(x)$-complete. 
\end{proof}

\begin{proposition}\label{prop:compformula}
Let $i_\land \colon \C^{\Cpl(I)} \subset \C$ denote the inclusion functor. Then $i_\land i^\land \simeq (-)^{i_\lor i^\lor A}$.
\end{proposition}

\begin{proof}
Cotensoring with the exact triangle
$
i_\vee i^\vee A\to A\to j_* j^* A
$
induces an exact triangle of endofunctors
$
(-)^{j_*j^*A}\to\id\to(-)^{i_\lor i^\lor A}.
$
Applying the completion endofunctor $i_\land i^\land$, we obtain an exact triangle of endofunctors
\[
i_\land i^\land((-)^{j_*j^*A})\to i_\land i^\land\to i_\land i^\land((-)^{i_\lor i^\lor A}).
\]
We claim that the endofunctor $i_\land i^\land((-)^{j_*j^*A})$ is null and that
the endofunctor $(-)^{i_\lor i^\lor A}$ factors through the full subcategory of $I$-complete objects, so that
\[
i_\land i^\land ((-)^{i_\lor i^\lor A})\simeq (-)^{i_\lor i^\lor A},
\]
from which it follows immediately that
$
i_\land i^\land\simeq (-)^{i_\lor i^\lor A}.
$
To see this, let $L$, $M$ and $N$ be objects of $\C$ such that $L$ is $I$-local and $N$ is $I$-nilpotent.
Then
\[
\Map(N,M^{j_* j^* A})\simeq\Map(j_* j^* A\otimes_A N,M)\simeq 0
\]
since $j_* j^* A\otimes_A N\simeq j_*j^* N\simeq 0$.
It follows that $M^{j_* j^* A}$ is $I$-local, so that $i_\land i^\land(M^{j_*j^*A})\simeq 0$.
Similarly,
\[
\Map(L,M^{i_\lor i^\lor A})\simeq\Map(i_\lor i^\lor A\otimes_A L,M)\simeq 0
\]
since $i_\lor i^\lor A\otimes_A L\simeq 0$ as $i_\lor i^\lor A$ is $I$-nilpotent.
It follows that $M^{i_\lor i^\lor}$ is $I$-complete, so that $i_\land i^\land(M^{i_\lor i^\lor A})\simeq M^{i_\lor i^\lor A}$, completing the proof.
\end{proof}
The following proposition is a special case of \cite[Proposition 7.3.1.7]{SAG}.
\begin{proposition}\label{prop:nilcompeq}
	Let $\C$ be a $A$-linear $\i$-category. The functors $ i^\land i_\lor$ and  $ i^\lor i_\land$ induce inverse equivalences
	$\C^{\nil(I)}\rightleftarrows \C^{\Cpl(I)}$.
\end{proposition}

\begin{lemma}\label{lem:tensornil}
Let $\C$ and $\D$ be (stable) $A$-linear $\i$-categories. Let $C$ be an $I$-nilpotent object of $\C$ and $D$ an arbitrary object of $\D$. Then the object $C \otimes_{A} D$ of $\C \otimes_{\Der(A)} \D$ (the image of $(C,D)\in\C\times\D\to\C\otimes_{\Der(A)} \D$) is $I$-nilpotent. 
\end{lemma}

\subsection{Completions of linear \texorpdfstring{$\i$}{infinity}-categories}

\begin{definition}
Let $A$ be a connective commutative ring spectrum.
\begin{enumerate}
\item A stable $A$-linear $\infty$-category $\C$ is $I$-nilpotent if every object of $\C$ is $I$-nilpotent.
\item A stable $A$-linear $\infty$-category $\C$ is $I$-local if every object of $\C$ is $I$-local.
\item A stable $A$-linear $\infty$-category $\C$ is $I$-complete if every object of $\C$ is $I$-complete.
\end{enumerate}
\end{definition}

\begin{lemma}\label{lem:cc}
	Let $\C\in\Mod_{\Der(A)}$.
	If $M,N\in\C^{\Cpl}$ then $\underline{\Map}(M,N)\in\Mod_A^{\Cpl}$.
	Moreover, if $N\in\C$ has the property that $\underline{\Map}(M,N)\in\Der_{\Cpl(I)}(A)$ for all $M\in\C$, then $N\in\C^{\Cpl}$.
\end{lemma}

\begin{proof}
Let $L\in\Der(A)^{\loc(I)}$, $M, N \in \C$. Then
\begin{align*}
\Map_A(L,\uMap_\C(M,N))\simeq \Map_\C(L \otimes M,N) \simeq \Map_\C(M,N^L).
\end{align*}	
Remark that $L \otimes M$ is an $I$-local object of $\C$. If $N \in \C^{\Cpl(I)}$, then $\Map_\C(L \otimes M,N)$ is contractible as well as $\Map_A(L,\uMap(M,N))$. This implies that $\uMap(M,N)$ is $I$-complete.
If  $\uMap(M,N)$ is $I$-complete for every $M \in \C$, then for every $x \in I$ $\Map_A(A[x^{-1}],\uMap(M,N))$ is contractible then $\Map_\C(M,N^{A[x^{-1}]})$ is contractible for every $M \in \C$. Hence, $N^{A[x^{-1}]} \simeq 0$ for every $x \in I$ which proves the claim.
\end{proof}
\begin{proposition}\label{prop:carcatcc}
Let $\C$ be a stable $A$-linear $\i$-category.
The following statements are equivalent:
\begin{enumerate}
\item \label{item:cc}
$\C$ is cohomologically complete.
\item \label{item:map}
For every $M$ and $N$ in $\C$, $\underline{\Map}(M,N)\in\Der_\Cpl(A)$.
\item \label{item:nilcomp}
$\C$ is $I$-nilpotent.
\item \label{item:facto} The action $\Der(A) \tensor \C \to \C$ factors through the map $i^\land\otimes\id:\Der(A) \tensor \C\to \Der_\Cpl(A) \tensor \C$.
\item \label{item:action}
$\C\in\Mod_{\Der_\Cpl(A)}$.
\end{enumerate}
\end{proposition}

\begin{proof}
\ref{item:cc} $\Leftrightarrow$ \ref{item:map} is Lemma \ref{lem:cc}.
To see that \ref{item:cc} $\Leftrightarrow$ \ref{item:nilcomp}, observe that the $\i$-category $\C$ is $I$-cohomologically complete if and only if $\C=\C^{\Cpl(I)}$. As the pair $(\C^{\loc(I)}, \C^{\Cpl(I)})$ is a semi-othogonal decomposition of $\C$, this is equivalent to $\C^{\loc}=0$. Since $(\C^{\nil(I)},\C^{\loc(I)})$ is also a semi-orthogonal decomposition of $\C$, this is equivalent to $\C=\C^{\nil(I)}$.
The implications \ref{item:facto}$\Leftrightarrow$ \ref{item:action} are clear.

To see that \ref{item:nilcomp} $\Rightarrow$ \ref{item:facto}, observe that an $A$-linear functor between $A$-linear $\i$-categories carries $I$-local object to $I$-local objects (\cite[Proposition 7.2.4.9 (iv)]{SAG}). If $V$ is an $I$-local object of $\Der(A)$, it follows that for every $M$ in $\C$, $V \otimes M$ is an $I$-local objects of $\C$. If $\C$ is nilpotent this implies that  $V \otimes M \simeq 0$. Hence, the action
$
 \Der(A) \otimes \C \to \C
$
vanishes on $\Der(A)^{\loc(I)} \otimes \C$. Moreover, $\Der_\Cpl(A)$ is canonically equivalent to the Verdier quotient $\Der(A) / \Der(A)^{\loc(I)}$. It follows from the universal property of the Verdier quotient that the action map $ \Der(A) \otimes \C \to \C$ factors through the completion map $\Der(A) \otimes \C {\rightarrow} \Der_{\Cpl(I)}(A) \otimes \C$.

Lastly we show that \ref{item:facto} $\Rightarrow$ \ref{item:cc}. Let $M$ and $N$ be object of $\C$ and $x \in I$. We have
\begin{align*}
\Map_\C(N,M^{A[x^{-1}]}) \simeq \Map_\C(A[x^{-1}] \otimes N,M) 
\end{align*}
and the factorization of the action $\Der(A) \otimes \C \to \C$ implies that it vanishes on $\Der(A)^{\loc(I)}$. As $A[x^{-1}]$ is $I$-local, $A[x^{-1}] \otimes N \simeq 0$. Hence, for every $x \in I$, $M^{A[x^{-1}]} \simeq 0$. By Proposition \ref{prop:complcrit}, $M$ is $I$-complete.
\end{proof}

\begin{proposition}
Let $\D$ be a presentable symmetric monoidal stable $\i$-category and
\[
\A\overset{j_*}{\to}\B\overset{k_*}{\to}\C
\]
a Verdier sequence of presentable $\D$-module $\i$-categories such that $j_*$ admits a left adjoint $j^*:\B\to\A$.
Then, for any $\D$-module $\M$,
\[
\A\otimes_\D\M\overset{j_*\otimes_\D\M}{\to}\B\otimes_\D\M\overset{k_*\otimes_\D\M}{\to}\C\otimes_\D\M
\]
is a Verdier sequence of presentable $\D$-module $\i$-categories.
\end{proposition}

\begin{proof}
Since $j_*$ is also a right adjoint, the functor $j_*\otimes_\D\M$ is calculated as postcomposition
\[
\Funr_{\D}(\M^{\op},\A)\to\Funr_{\D}(\M^{\op},\B)
\]
with $j_*:\A\to\B$, which is fully faithful since $\A\to\B$ is fully faithful.
Here
\[
\Funr_{\D}(\M^{\op},\A)\simeq\lim\left\{\Funr(\M^{\op},\A)\rrarrow\Funr(\M^{\op},\Funr(\D^{\op},\A))\rrrarrow\cdots\right\}
\]
denotes the totalization of the action of $\D$ on $\Funr(\M^{\op},\A)$; that is, the cosimplicial diagram obtained from the simplicial diagram realizing the relative tensor product
\[
\A\otimes_\D\M\simeq\colim\left\{\A\otimes\M\llarrow\A\otimes\D\otimes\M\lllarrow\cdots\right\}
\]
in $\PrL$.
Since $\C$ is the cofiber of the inclusion $j_*:\A\to\B$, and tensoring with any $\D$-module preserves colimits and zero objects, we obtain a cofiber sequence
\[
\A\otimes_\D\M\to\B\otimes_\D\M\to\C\otimes_\D\M
\]
in $\PrL$, as desired.
\end{proof}

\begin{corollary}
Suppose that $\A\overset{j_*}{\to}\B\overset{k_*}{\to}\C$ is a Verdier sequence of presentable $\D$-module $\i$-categories such that $j_*$ admits a left adjoint and $\A\otimes_\D\C\simeq 0$.
Then the canonical map
\[
\B\otimes_\D\C\to\C\otimes_\D\C
\]
is an equivalence of presentable $\D$-module $\i$-categories.
If additionally $\B\simeq\D$, we obtain an equivalence
\[
\C\simeq\D\otimes_\D\C\overset{k_*\otimes_\D\C}{\to}\C\otimes_\D\C.
\]
In particular, $\C$ inherits the structure of an idempotent commutative $\D$-algebra.
\end{corollary}

\section{Model structures}\label{sec:modstruct}

\subsection{The semi-free model structure}

We use the notion of $(\GGG, \HHH)$-descent structure introduced by D.-C. Cisinsky an F. D\'eglise in \cite{Cisinski2009} to endow $\Mod_{\A}$, the category of modules over an algebroid stack, with a model structure which we call the semi-free model structure. The technology of descent structure allows to specify the model structure ``locally'' and allows to bypass the fact that we are working over an algebroid. Then, the verification required to apply the technique of descent structures reduces to questions of classical sheaf theory already addressed in \cite{Cisinski2009}. 
We  use the semi-free model structure to derive various tensor products on categories of DQ-modules.

Let $X$ be a topological space endowed with a $k$-algebroid stack $\A$. We freely use the notion of gluing datum for algebroids in what follows and refer to \cite[\textsection 2.1]{KS3} for details. By definition of an algebroid stack, we can find a covering $\U=\{ U_i \}_{i \in I}$ of $X$ such that for every $i \in I$, there are

\begin{equation*}
\begin{cases}
\sigma_i \in \A(U_i),\\
\textnormal{isomorphisms} \; \phi_{ij} \colon \sigma_j|_{U_{ij}} \stackrel{\sim}{\longrightarrow} \sigma_j|_{U_{ij}}.
\end{cases}
\end{equation*}
We associate to this data the following:
\begin{itemize}
\item for every $i \in I$, the sheaf of $k$-algebras $\A_i=\fEnd(\sigma_i)$ on $U_i$,
\item the $k$-algebra isomorphisms $f_{ij} \colon \A_j|_{U_{ij}} \stackrel{\sim}{\longrightarrow} \A_i|_{U_{ij}}$ induced by the $\phi_{ij}$.
\end{itemize}
We call such a covering $\U$ a trivializing covering for $\A$. 

Let $\sMod_{\A}$ be the stack of modules over the algebroid $\A$ that is for every open set $U \subset X$, $\sMod_{\A}(U):=\Mod_{\A|_U}$ and let $\sMod_{\A_i}$ be the stack of modules over the sheaf of algebras $\A_i$. Then there is an equivalence of stacks
\begin{equation*}
\phi_i \colon \sMod_{\A_i}  \stackrel{\sim}{\longrightarrow} {\sMod_{\A}}|_{U_i}.
\end{equation*}

For every $V \subset U_i \in \U$, we consider the sheaf $\A_{i,V}$ defined by 
\begin{equation*}
\A_{i,V}:={j^{-1}_{V \to U_i}}\,{j_{V \to U_i\, *}} \A_i
\end{equation*}
where $j_{V \to U_i}\colon U_i\to V$ is the morphism of sites defined by $V^\prime \subset V \mapsto V^\prime \subset U_i$.
The set
\begin{equation*}
\GGG_\U(\A)=\mathfrak\lbrace{ \phi_i(\A_{i, V}) | \; i \in I, V \subset U_i \; \textnormal{and} \; U_i \in \U} \rbrace
\end{equation*}
is a set of generators of $\Mod_{\A}$. This follows from the definition of $\Mod_{\A}$ which is defined as the category of $k$-enriched functors $\Fun_k(\A, \sMod_{k_X})$. If there is no risk of confusion, we simply write $\GGG$ instead of $\GGG_\U(\A)$.

\begin{remark}
If their is no risk of confusion, we will omit to write the equivalences $\phi_i$. For instance, we will often write $\A_{i, V}$ instead of  $\phi_i(\A_{i, V})$.
\end{remark}
Let $\M \in \Mod_{\A}$, we write $D^n \M$ for the acyclic complex such that 
$
(D^n \M)^n=(D^n \M)^{n+1}=\M
$
and zero otherwise and whose only non-zero differential is the identity of $\M$. We denote by $S^n \M$, the complexe in degree $n$ such that $(S^n \M)^n=\M$.
Adapting \cite[Example 2.3]{Cisinski2009}, we define the class of $\GGG$-cofibrations to be the smallest class of maps in $\Ch(\A)$ stable by pushouts, transfinite compositions, retracts and generated by the inclusions $i_{n,\G} \colon S^{n+1} \G \to D^n \G$ for every $n \in \ZZ$ and $\G \in \GGG$.

Consider an open set $V$ of $X$ such that there exists $U_i \in \U$ with $V \subset U_i$ and let $\underline{V}$ be an hypercover of $V$. Since $V \subset U_i$, we work in $\Mod_{\A_i}$. Note that the choice of the $U_i$ containing $V$ (and hence of the $\A_i$) is irrelevant because the different possible choices lead to equivalent construction as it can be shown using the isomorphisms of algebras $\phi_{ij}$ provided by the gluing datum. 

The simplicial $\A_i$-module $\mathsf{S}( \A_{i})_{\underline{V}}$, freely generated by the hypercover $\underline{V}$, correponds to a complex (the Moore complex of $\mathsf{S}( \A_{i})_{\underline{V}}$) of $\A_i$-modules denoted $\A_{i, \underline{V}}$ such that
$
(\A_{i,\underline{V}})^{-n}:=\mathsf{S}(\A_{i})_{\underline{V}, \,n}
$
and the differential are given by the alternated sum of the face operators (see \cite[Expos\'e V]{SGA4} for details).
This complex comes with a canonical map
$
\A_{i, \underline{V}} \to \A_{i, V}
$
which is a quasi-isomorphism.
We consider the mapping cone of this map and obtain an object of $\Ch(\A)$ that we denote $\widetilde{\A}_{i,\underline{V}}$. We define $\HHH_\U(\A)$ as the family of all complexes of the form $\widetilde{\A}_{i,\underline{V}}$ for any hypercover of an open subset $V$ of $X$ such that there exists $U_i \in \U$ such that $V \subset U_i$.

\begin{proposition}
The pair $(\GGG_\U(\A),\HHH_\U(\A))$ is a descent structure. 
\end{proposition}

\begin{proof}
The proof is similar to the one for the descent structure of \cite[Example 2.3]{Cisinski2009}.
\end{proof}

\begin{proposition}
The descent structure $(\GGG_\U(\A),\HHH_\U(\A))$ induces a proper cellular model structure on $\Ch(\A)$ in which the weak equivalences are the quasi-isomorphisms and the cofibrations are the $\GGG_\U(\A)$-cofibrations. We call this model structure the semi-free model structure subordinated to $\U$ and denote by $\Ch(\A)_{\smf(\U)}$ the category $\Ch(\A)$ endowed with this model structure.
\end{proposition}

\begin{proof}
This follows from \cite[Theorem 2.5]{Cisinski2009}. 
\end{proof}

\begin{remark}
Assuming the algebroid $\A$ is flat over $k$, it is straightforward to check that $\Ch(\A)_{\smf(\U)}$ is a $\Ch(k)$-enriched model structure when $\Ch(k)$ is endowed with the projective model structure.
\end{remark}

\subsection{Deriving the \texorpdfstring{$\gr$}{gr} functor}\label{subsec:gr}

Let $\A_X$ be a DQ-algebroid stack. We consider the totalization of the functor \eqref{def:grDQ} and get
\begin{equation*}
\gr \colon \Ch(\A_X) \to \Ch(\gr \A_X), \quad \M \mapsto \gr \A_X \tensor_{\A_X} \M.
\end{equation*}
We choose a trivializing covering $\U$ for $\A_X$, hence also for $\gr \A_X$. We endow $\Ch(\A_X)$ and $\Ch(\gr \A_X)$ with their respective semi-free model structures subordinated to $\U$.

It follows from the definition of $\gr$ that it preserves colimits. It clearly takes elements of $\GGG_\U(\A_X)$ to elements of $\GGG_\U(\gr \A_X)$. Let us check that the $\gr$ functor takes elements of $\HHH_\U(\A_X)$ to elements of $\HHH_\U(\gr \A_X)$. 
The elements of $\HHH_\U(\A_X)$ are mapping cones of morphisms of the form 
\begin{equation} \label{comp:hypres}
\A_{i, \underline{V}} \to \A_{i,V}
\end{equation}
where $V$ is an open subset of some $U_i \in \U$, $\A_{i}$ is the restriction of $\A_X$ to $U_i$ (an we identify $\A_i$ to a sheaf of $\Chbar$-algebras), $\underline{V}$ is an hypercovering of $V$ and $\A_{i, \underline{V}}$ is the Moore complex associated to the simplicial $\A_i$-module freely generated by the hypercover $\underline{V}$.  By construction $\A_{i, \underline{V}}$ is a complex such for $n\leq 0$, $(\A_{i, \underline{V}})^n$ is a direct sum of element of $\GGG_\U(\A_X)$ and zero if $n > 0$. This implies that $\A_{i,\underline{V}}$ is a flat complex over $\A_i$ as well as $\A_{i,V}$. Hence, applying the functor $\gr$ to the morphism \eqref{comp:hypres}, we get a map
\begin{equation*}
\gr(\A_{i})_{\underline{V}} \to \gr(\A_{i})_{V}
\end{equation*} 
and the mapping cone of this map belongs to $\HHH_\U(\gr\A_X)$.
Applying \cite[Theorem 2.14]{Cisinski2009}, we obtain the following result.

\begin{proposition}
The functors
$
\gr \colon \Ch(\A_X) \rightleftarrows \Ch(\gr \A_X) \colon \iota
$
form a Quillen adjunction.
\end{proposition}

The category $\Ch(\gr \A_X)$ is a $\Ch(\CC)$-module; that is, it is tensored over $\Ch(\CC)$. Hence, it is a $\Ch(\Chbar)$-module via the symmetric monoidal functor $\gr \colon \Ch(\Chbar) \to \Ch(\CC)$ which sends the complex $M$ to the complex $\CC \otimes_{\Chbar} M$.
This induces a $\Ch(\Chbar)$-module structure on $\Ch(\gr \A_X)$ given by
\begin{equation*}
\odot \colon \Ch(\Chbar) \times \Ch(\gr \A_X) \to \Ch(\gr \A_X), \quad (M, \F) \mapsto M \odot \F := \gr(M) \tensor \F. 
\end{equation*}

\begin{proposition}
The functor $\odot$ is a Quillen bifunctor when $\Ch(\Chbar)$ is endowed with the projective model structure and $\Ch(\gr \A_X)$ is endowed with the semi-free model structure subordinated to $\U$.
\end{proposition}

\begin{proof}
This follows by applying \cite[Lemma 4.2.4]{Hovey99} to the generating cofibrations of $\Ch(\Chbar)$ and $\Ch(\gr \A_X)$. 
\end{proof}

Hence, we obtain a morphism $\gr : \Der(\A_X) \to \Der(\gr \A_X)$ in $\Mod_{\Der(\Chbar)}$.

\subsection{Tensor products}\label{subsec:tensors}
As it is the case for abelian sheaves, deriving right exact functor is more delicate than deriving left exact functor. Hence, to derive the various tensor products of DQ-modules, we rely on the theory of $(\GGG,\HHH)$-descent structures of \cite{Cisinski2009}. Using their methods we have defined a model structure that allow us to derive the tensor products of DQ-modules. In what follow, we use the following notational convention. 

\begin{notation}\label{not:spaceconvention}
\begin{enumerate}
\item
If $X$ is a smooth complex variety endowed with a DQ-algebroid $\A_X$, we denote by $X^\mathrm{a}$ the same variety endowed with the opposite DQ-algebroid $\A_X^{\op}$ and we write $\A_{X^\mathrm{a}}$ for this algebroid.
\item
Consider a product of smooth complex varieties $X_1\times X_2 \times X_3$, we write it $X_{123}$. We denote by
$p_i$ the $i$-th projection and by $p_{ij}$ the $(i,j)$-th projection
({\em e.g.,} $p_{13}$ is the projection from 
$X_1\times X_1^\mathrm{a}\times X_2$ to $X_1\times X_2$). 
 \item
We write $\A_i$ and $\A_{ij^\mathrm{a}}$
instead of $\A_{X_i}$ and $\A_{X_i\times X_j^\mathrm{a}}$  and
similarly with other products.\\
\end{enumerate}
\end{notation}
Let $X_i$ $(i=1,\;2,\;3)$ be smooth complex varieties endowed with  DQ-algebroid stacks $\A_i$.

\begin{definition}\label{def:modtens}
We define a functor
$
- \otimes_{\A_2} - \colon \Mod_{\A_{12^\mathrm{a}}} \times \Mod_{\A_{23^\mathrm{a}}} \to \Mod_{\A_1 \btimes \CC^\hbar_2 \btimes \A_3}
$
by the rule $(\K_1, \K_2) \mapsto p_{12}^{-1} \K_1 \tensor_{p_{2}^{-1}\A_{2}}  p_{23}^{-1} \K_2$.
\end{definition}

The tensor product of Definition \ref{def:modtens} is not always adapted to the study of DQ-modules. Hence, the following variant has been introduced in \cite{KS3}.

\begin{definition}\label{def:DQtens}\cite[Definition 3.1.3]{KS3} We write $(-)\utensor_{\A_2} (-) \colon \Mod_{\A_{12^\mathrm{a}}} \times \Mod_{\A_{23^\mathrm{a}}} \to \Mod_{p_{13}^{-1}\A_{13^\mathrm{a}}}$ for the functor $(\K_1, \K_2) \mapsto p_{12}^{-1} \K_1 \tensor_{p_{12}^{-1}\A_{12^\mathrm{a}}} \A_{123} \tensor_{p_{23^\mathrm{a}}^{-1} \A_{23^\mathrm{a}}} p_{23}^{-1} \K_2$.
\end{definition}

The functors of Definitions \ref{def:modtens} and \ref{def:DQtens} are extended to categories of complexes via totalization. In what follows, we derive both functors and show that the derived functor of $\utensor_{\A_2}$ can be expressed in term of the derived functor of $\tensor_{\A_2}$. Both functors commutes with colimits in each variables separately.

\begin{proposition}
The functors $\tensor_{\A_2}$ and $\utensor_{\A_2}$ are adjunctions of two variables.
\end{proposition}

\begin{lemma}\label{lem:flatness_preservation}
Let $\M\in\Ch(\A_{12^\mathrm{a}})$ be flat. Then $\M \utensor_{\A_2} p_{23}^{-1} (-) \colon \Ch(\A_{23^\mathrm{a}}) \to \Ch(p_{13}^{-1}\A_{13^\mathrm{a}})$ preserves acyclic complexes.
\end{lemma}

\begin{proof}
The question is local. Thus, we can assume that $\A_1$, $\A_2$ and $\A_3$ are DQ-algebras. Since $\A_{123}$ is flat over $p_{23}^{-1}\A_{23^\mathrm{a}}$, it follows that for every acyclic complex $\N \in \Ch(\A_{23^\mathrm{a}})$, $\A_{123} \otimes_{p_{23}^{-1}\A_{23^\mathrm{a}}} p_{23}^{-1} \N$
is acyclic.
Since $\M$ is flat over $\A_{12^\mathrm{a}}$, $p_{12}^{-1} \M \otimes_{p_{12}^{-1} \A_{12^\mathrm{a}}} \A_{123} \otimes_{p_{23}^{-1}\A_{23^\mathrm{a}}} p_{23}^{-1} \N$
is again acyclic.
\end{proof}

For $i=1, \; 2 \; ,3$, we let $\U_i$ be a covering of $X_i$ such that for every $U \in \U_i$, $\A_i |_U$ is trivial. We write $\U_{ij}$ for the covering of $X_i \times X_j$ whose opens are the $U_\alpha \times U_\beta$ with $U_\alpha \in \U_i$ and $U_\beta \in \U_j$. As in \cite{Cisinski2009}, we follow the notation of \cite{Hovey99}. 
Recall that the cylinder $\Cyl(\M)$ of a complex $\M$ is defined by
$
\Cyl(\M)^j= \M^j \oplus \M^{j+1} \oplus \M^{j}
$
with differential $d(x,y,z)=(dx-y,-dy,y+dz)$. Moreover, the map  $a:\Cyl(\M) \to \M$ which sends $(x, y ,z)$ to $x+z$
is a quasi-isomorphism and there is a canonical inclusion 
$
(i_0, i_1) \colon \M \oplus \M \to \Cyl(\M)
$
where $i_0(x)=(x,0,0)$ and $i_1(z)=(0,0,z)$.

We know by \cite[\textsection 1.9]{Cisinski2009} that 

\begin{itemize}

\item a set of generating cofibrations denoted $I_{ \U_{i(i+1)} }$ ($i= 1, \;2)$ is the set of inclusions of the form $S^{n+1} \G \to D^n \G$ for any $n \in \ZZ$ and $\G$ in $\GGG_{\U_{i(i+1)}}(\A_{(i(i+1)})$. We write $I_{i(i+1)}$-$cof$ for the set of $\GGG_{\U_{i(i+1)}}(\A_{i(i+1)})$-cofibrations.

\item a set of generating trivial cofibrations is $J_{ \U_{i(i+1)}}=J_{ \U_{i(i+1)}}^\prime \cup J_{ \U_{i(i+1)}}^{\prime \prime}$ where $J_{ \U_{i(i+1)}}^\prime$ is the set of maps $0 \to D^n \G$ for $n \in \ZZ$ and $\G \in \GGG_{\U_{i(i+1)}}(\A_{(i(i+1)})$ and $J_{ \U_{i(i+1)}}^{\prime \prime}$ the set of maps $\H \oplus \H [n] \to \Cyl(\H)[n]$ for any $n \in \ZZ$ and any $\H \in \HHH_{\U_{i(i+1)}}(\A_{i(i+1)})$.
\end{itemize}

\begin{lemma}\label{lem:flatgenerator}
All the $\G \in \GGG_{\U_{i(i+1)}}(\A_{(i(i+1)})$ ($i=1, \,2$) are flat for $- \tensor_{\A_2} -$ and $- \utensor_{\A_2} -$.
\end{lemma}

\begin{proof}
It follows from \cite{KS3} that $\A_{12^\mathrm{a}}$ is flat over $p_2^{-1}\A_2$. Since the elements of $\GGG_{\U_{12}}(\A_{12^\mathrm{a}})$ are of the form $\A_{12^\mathrm{a} , V}$ with $V$ an open subset of some $U \in \U_{12}$, they are flat over $p_2^{-1} \A_2$.  
\end{proof}
 
\begin{remark}
All the $\H \in \HHH_{\U_{i(i+1)}}(\A_{i(i+1)})$ are acyclic as they are cones of quasi-isomorphisms.
\end{remark}

In what follows we endow $\Ch(\A_{i(i+1)^\mathrm{a}})$ with the semi-free model structure subordinated to $\U_i \times \U_{i+1}$ $(i=1, \;2)$, and $\Ch(\A_1 \btimes \CC^\hbar_2 \btimes \A_3)$ with the injective, model structures.

\begin{proposition}
The functor $- \tensor_{\A_2} - \colon \Ch(\A_{12^\mathrm{a}}) \times \Ch(\A_{23^\mathrm{a}}) \to \Ch(\A_1 \btimes \CC^\hbar_2 \btimes \A_3)$ given by $(\K_1, \K_2) \mapsto p_{12}^{-1} \K_1 \tensor_{p_{2}^{-1}\A_{2}}  p_{23}^{-1} \K_2$ is a left Quillen bifunctor.
\end{proposition}

\begin{proof}
The proof of this proposition is similar that of Proposition \ref{prop:derDQtensor} below. The essential difference is that one only need Lemma \ref{lem:flatgenerator} and not Lemma \ref{lem:flatness_preservation} and Lemma \ref{lem:flatgenerator}.
\end{proof}

\begin{proposition}\label{prop:derDQtensor}
The functor $- \utensor_{\A_2} -\colon \Ch(\A_{12^\mathrm{a}}) \times \Ch(\A_{23^\mathrm{a}}) \to \Ch(p_{13}^{-1}\A_{13^\mathrm{a}\\
})$ given by $(\K_1, \K_2) \mapsto p_{12}^{-1} \K_1 \tensor_{p_{12}^{-1}\A_{12^\mathrm{a}\\
}} \A_{123} \tensor_{p_{23^\mathrm{a}}^{-1} \A\\
_{23}} p_{23}^{-1} \K_2$ is a left Quillen bifunctor.
\end{proposition}

\begin{proof}
By \cite[Definition 4.2.1]{Hovey99}, it suffices to show that if $f \colon \M \to \N$ is a cofibration in $\Ch(\A_{12^\mathrm{a}})_{\smf(\U_{12})}$ and $g \colon \S \to \T$ is a cofibration in $\Ch(\A_{23^\mathrm{a}})_{\smf(\U_{23})}$, the induced map $f \, \square \, g \colon \P \to \N \utensor_{A_2} \T$ (where $\P$ is the pushout defined in the diagram below) is a cofibration in $\Ch(p_{13}^{-1}\A_{13^\mathrm{a}})_\inj$, which is trivial if either $f$ or $g$ is trivial.

\begin{equation*}
\xymatrix{ 
\M \utensor_{\A_2} \S \ar[r]^-{\id \utensor g} \ar[d]_-{f \utensor \id} 
& \M \utensor_{\A_2} \T \ar[d]  \ar@/^/[ddr]^-{f \utensor \id}
&
\\
\N \utensor_{\A_2} \S  \ar[r]   \ar@/_/[drr]_-{\id \utensor g}
&  \P  \pullbackcorner[lu] \ar[rd]^-{f \, \square \, g} 
& 
\\
&
& \N \utensor_{\A_2} \T.
}
\end{equation*}
Since $- \utensor_{\A_2} -$ is an adjunction of two variables and the categories $\Ch(\A_{i(i+1)})_{\smf(\U_{i(i+1))}}$ $(i=1, 2)$ are cofibrantly generated, it is enough by  \cite[Corollary 4.2.5]{Hovey99} to check that
\begin{enumerate}
\item[(1)] $f \square g$ is a cofibration (i.e. a monomorphism) when $f \in I_{\U_{12}}$ and $g \in I_{\U_{23}}$.
\item[(2)] $f \square g$ is a trivial cofibration (i.e. a monomorphism and a quasi-isomorphism) when $f \in J_{\U_{12}}$ and $g \in I_{\U_{23}}$ or when $f \in I_{\U_{12}}$ and $g \in J_{\U_{23}}$.
\end{enumerate}

For (1), let $\E_V$ in $\GGG_{\U_{12}}(\A_{12^\mathrm{a}})$ and $\F_W$ in $\GGG_{\U_{23}}(\A_{23^\mathrm{a}})$.
We consider the map
\begin{align*}
i_{n,\E_V} \colon S^{n+1} \E_V \to D^{n} \E_V && \textnormal{and} && i_{p,\F_W} \colon S^{p+1} \F_W \to D^{p} \F_W.
\end{align*}
Then a direct computation shows that the pushout of the diagram
\begin{equation*}
D^n \E_V \utensor_{\A_2} S^{p+1} \F_W\leftarrow S^{n+1} \E_V \utensor_{\A_2} S^{p+1} \F_W \rightarrow S^{n+1} \E_V \utensor_{\A_2} D^p \F_W
\end{equation*}
is the complex $\P$ of $\Ch(p_{13}^{-1} \A_{13^\mathrm{a}})$ concentrated in (cohomological) degree $n+p+1$ and $n+p+2$
\begin{equation*}
\E_V \utensor_{\A_2} \F_W \oplus \E_V \utensor_{\A_2} \F_W \stackrel{+}{\longrightarrow} \E_V \utensor_{\A_2} \F_W
\end{equation*}
where $+$ denotes the fold map. Moreover, the map of complexes
\begin{equation*}
i_{n, \E_V} \, \square \, i_{p, \F_W} \colon \P \to D^n \E_V \utensor _{\A_2} D^p \F_W.
\end{equation*}
is zero everywhere but in degree $n+p+1$ and $n+p+2$ where it is the identity. Hence, it is a monomorphism.

To verify (2), we have to treat two cases. 
Let $\E_V \in \GGG_{\U_{12}}(\A_{12^\mathrm{a}})$ and $\H \in \HHH_{\U_{23^\mathrm{a}}}(\A_{23^\mathrm{a}})$.
\begin{enumerate}[label=(\roman*)]
\item We consider the case where the morphisms are 
\begin{align*}
j_n \colon 0 \to D^n \E_V && \textnormal{and}  &&(i_0,i_1)[p] \colon \H \oplus \H[p] \to \Cyl(\H)[p].
\end{align*}
Then,
\begin{equation*}
j_n \, \square \, (i_0,i_1)[p] \colon D^n \E_V \utensor_{\A_2} \H \oplus \H[p] \longrightarrow  D^n \E_V \utensor_{\A_2} \Cyl(\H)[p]
\end{equation*} 
is equal to $\id_{D^n \E_V} \utensor_{\A_2} (i_0,i_1)[p]$. Hence, it is a monomorphism. 
We need to check that the map $j_n \, \square (i_0,i_1)[p]$ is a quasi-isomorphism. Since $\H$ is acyclic $\Cyl(\H)$ is also acyclic. Moreover, $\E_V$ is flat over $\A_{12^\mathrm{a}}$. Thus, it follows from Lemma \ref{lem:flatness_preservation} and \ref{lem:flatgenerator} that
$D^n \E_V \utensor_{\A_2} \H \oplus \H[p]$ and $D^n \E_V \utensor_{\A_2} \Cyl(\H)[p]$ are acyclic. Hence, $j_n \, \square (i_0,i_1)[p]$ is a quasi-isomorphism.\\

\item We now consider the case where the morphisms are 
\begin{align*}
i_{n,\E_V} \colon S^{n+1} \E_V \to D^n \E_V && \textnormal{and} && (i_0,i_1)[p] \colon \H \oplus \H[p] \to \Cyl(\H)[p].
\end{align*}
We need to check that the map $i_{n,\E_V} \, \square \, (i_0,i_1)[p]$ is a monomorphism. We proceed as follow.

We consider the morphism
\begin{equation*}
\xymatrix{
\pi \colon S^{n+1} \E_V \utensor_{\A_2} \H \oplus \H[p] \ar[r]
& D^n \E_V \utensor_{\A_2} \H \oplus \H[p] \bigoplus S^{n+1}\E_V \utensor_{\A_2} \Cyl(\H)[p] 
}
\end{equation*}
where
\begin{equation*}
\pi=(i_{n,\E_V} \utensor_{\A_2} \id_{\H \oplus \H}, -\id_{S^{n+1} \E_V} \utensor_{\A_2} (i_0,i_1)[p]).
\end{equation*}
A direct computation shows that $\pi$ is a monomorphism. Moreover, the cokernel $\P$ of $\pi$ is the pushout of the diagram
\begin{equation*}
D^n \E_V \utensor_{\A_2} \H \oplus \H[p]\leftarrow S^{n+1} \E_V \utensor_{\A_2} \H \oplus \H[p] \rightarrow S^{n+1} \E_V \utensor_{\A_2} \Cyl(\H)[p]
\end{equation*}
and $\pi$ is a kernel of the map
\begin{equation*}
\alpha \colon D^n \E_V \utensor_{\A_2} \H \oplus \H[p] \bigoplus S^{n+1}\E_V \utensor_{\A_2} \Cyl(\H)[p] \to D^n\E_V \utensor_{\A_2} \Cyl(\H)[p]
\end{equation*}
where 
\begin{equation*}
\alpha= \id_{D^n \E_V} \utensor_{\A_2} (i_0,i_1)[p] + i_{n , \E_V} \utensor_{\A_2} \id_{\Cyl(\H)[p]}.
\end{equation*}
Hence, the map $i_{n,\E_V} \, \square  \, (i_0,i_1)[p]$ is a monomorphism. We still need to check that it is a quasi-isomorphism. For that purpose, we show that both $\P$ and $D^n\E_V \utensor_{\A_2} \Cyl(\H)[p]$ are acyclic. The acyclicity of $D^n\E_V \utensor_{\A_2} \Cyl(\H)[p]$ has already been justified. Thus, we focus on the case of $\P$.
By definition of $\P$, we have the following exact sequence
\begin{equation*}
0 \to  S^{n+1} \E_V \utensor_{\A_2} \H \oplus \H[p] \stackrel{\pi}{\to} D^n \E_V \utensor_{\A_2} \H \oplus \H[p] \bigoplus S^{n+1}\E_V \utensor_{\A_2} \Cyl(\H)[p] \to \P \to 0
\end{equation*}
where the first two terms are acyclic by Lemma \ref{lem:flatness_preservation}. This implies that $\P$ is  acyclic.
\end{enumerate}
The remaining cases are analogous to those already addressed and are left to the reader.
\end{proof}

\begin{remark}
It is straightforward to check that $\tensor_{\A_2}$ and $\utensor_{\A_2}$ are $\Ch(\Chbar)$-linear Quillen bifunctors.
\end{remark}

It follows from the above propositions that we obtain the following pair of derived functors
\begin{align*}
- \Ltensor_{\A_2} - &\colon \Der(\A_{12^\mathrm{a}}) \times \Der(\A_{23^\mathrm{a}}) \to \Der(\A_1 \btimes \CC^\hbar_2 \btimes \A_3)\\
- \Lutensor_{\A_2} - &\colon \Der(\A_{12^\mathrm{a}}) \times \Der(\A_{23^\mathrm{a}}) \to \Der(p_{13}^{-1}\A_{13^\mathrm{a}}).
\end{align*}
If there is no risk of confusion we will write $\tensor_{\A_2}$ and $\utensor_{\A_2}$ instead of $\Ltensor_{\A_2}$ and $\Lutensor_{\A_2}$.
By \cite[Corollary 1.3.4.26]{HA}, $\Ltensor_{\A_2}$ and $\Lutensor_{\A_2}$ are $\CC^\hbar$-linear and commute with small colimits in each variables separately. 

\begin{proposition}\label{prop:underiving}
Let $\M$ be a $\GGG_{\U_{12}}(\A_{12^\mathrm{a}})$-cofibrant complex in $\Ch(\A_{12^\mathrm{a}})_{\smf(\U_{12})}$. Then $\M \tensor_{\A_2} (-)$ and $\M \utensor_{\A_2} (-) $ preserve quasi-isomorphisms.
\end{proposition}

\begin{proof}
The proof is similar to the proof of \cite[Proposition 3.7]{Cisinski2009}, which relies on \cite[Lemma 3.6]{Cisinski2009}, the proof of which carries over to our setting.
\end{proof}

\begin{corollary}
Let $\K_1 \in \Der(\A_{12^\mathrm{a}})$ and $\K_2 \in \Der(\A_{23^\mathrm{a}})$. Then 
\begin{equation*}
\K_1 \Lutensor_{A_2} \K_2 \simeq p_{12}^{-1} \K_1 \Ltensor_{p_{12}^{-1}\A_{12^\mathrm{a}}} \A_{123} \Ltensor_{p_{23^\mathrm{a}}^{-1} \A_{23}} p_{23}^{-1} \K_2.
\end{equation*}
\end{corollary}

\begin{proof}
Let $\Q_1$ be a $\GGG_{\U_{12}}(\A_{12^\mathrm{a}})$-cofibrant replacement of $\K_1$.
It follows from Proposition \ref{prop:underiving} that
\[
\K_1 \Lutensor_{A_2} \K_2
\simeq \Q_1 \utensor_{A_2} \K_2
\simeq p_{12}^{-1} \Q_1 \tensor_{p_{12}^{-1}\A_{12^\mathrm{a}}} \A_{123} \tensor_{p_{23^\mathrm{a}}^{-1} \A_{23}} p_{23}^{-1} \K_2
\simeq p_{12}^{-1} \K_1 \Ltensor_{p_{12}^{-1}\A_{12^\mathrm{a}}} \A_{123} \Ltensor_{p_{23^\mathrm{a}}^{-1} \A_{23}} p_{23}^{-1} \K_2.
\]
\end{proof}

\newcommand{\etalchar}[1]{$^{#1}$}
\providecommand{\bysame}{\leavevmode\hbox to3em{\hrulefill}\thinspace}
\providecommand{\MR}{\relax\ifhmode\unskip\space\fi MR }
\providecommand{\MRhref}[2]{%
	\href{http://www.ams.org/mathscinet-getitem?mr=#1}{#2}
}
\providecommand{\href}[2]{#2}

\Addresses
\end{document}